%% file: main.tex
\documentclass[a4paper,UKenglish,cleveref, autoref, thm-restate]{lipics-v2021}
\hideLIPIcs  %uncomment to remove references to LIPIcs series (logo, DOI, ...), e.g. when preparing a pre-final version to be uploaded to arXiv or another public repository

\input{imports.tex}

\title{Separability Properties of Monadically Dependent Graph Classes} %TODO Please add

\author{\'{E}douard Bonnet}{CNRS, ENS de Lyon, Université Claude Bernard Lyon 1, LIP UMR 5668, Lyon, France}{edouard.bonnet@ens-lyon.fr}{https://orcid.org/0000-0002-1653-5822}{}

\author{Samuel Braunfeld}{Computer Science Institute, Charles University, Prague, Czech Republic; and
the Czech Academy of Sciences, Institute of Computer Science, Pod Vod\'{a}renskou v\v{e}\v{z}\'{\i} 2, 182 00 Prague, Czech Republic}{sbraunfeld@iuuk.mff.cuni.cz}{https://orcid.org/0000-0003-3531-9970}{}

\author{Ioannis Eleftheriadis}{Department of Computer Science and Technology, University of Cambridge, UK}{ie257@cam.ac.uk}{https://orcid.org/0000-0003-4764-8894}{}

\author{Colin Geniet}{Discrete Mathematics Group, Institute for Basic Science (IBS), Daejeon, South Korea}{colin@ibs.re.kr}{https://orcid.org/0000-0003-4034-7634}{}

\author{Nikolas Mählmann}{University of Bremen, Germany}{maehlmann@uni-bremen.de}{https://orcid.org/0000-0003-3657-7736}{}

\author{Micha{\l} Pilipczuk}{University of Warsaw, Poland}{michal.pilipczuk@mimuw.edu.pl}{https://orcid.org/0000-0001-7891-1988}{}

\author{Wojciech Przybyszewski}{University of Warsaw, Poland}{przybyszewski@mimuw.edu.pl}{https://orcid.org/0000-0003-1158-9925}{}

\author{Szymon Toru{\'n}czyk}{University of Warsaw, Poland}{szymtor@mimuw.edu.pl}{https://orcid.org/0000-0002-1130-9033}{}

\authorrunning{Bonnet et al.} %TODO mandatory. First: Use abbreviated first/middle names. Second (only in severe cases): Use first author plus 'et al.'

\Copyright{\'{E}douard Bonnet, Samuel Braunfeld, Ioannis Eleftheriadis, Colin Geniet, Nikolas Mählmann, Micha{\l} Pilipczuk, Wojciech Przybyszewski, and Szymon Toru{\'n}czyk} %TODO mandatory, please use full first names. LIPIcs license is "CC-BY";  http://creativecommons.org/licenses/by/3.0/


 %TODO mandatory; please add comma-separated list of keywords

\category{Track B: Automata, Logic, Semantics, and Theory of Programming}

\relatedversion{} %optional, e.g. full version hosted on arXiv, HAL, or other respository/website
\funding{
    \'{E}B and CG were supported by the ANR projects TWIN-WIDTH (ANR-21-CE48-0014-01) and DIGRAPHS (ANR-19-CE48-0013-01).
    SB received funding from the European Research Council (ERC) under the European Union’s Horizon 2020 research and innovation programme (grant agreement No.\ 810115 - Dynasnet),
    and by Project 24-12591M of the Czech Science Foundation (GAČR), and supported partly by the long-term strategic development financing of the Institute of Computer Science (RVO: 67985807).
    IE was supported by a George and Marie Vergottis Scholarship awarded through Cambridge Trust, an Onassis Foundation Scholarship, and a Robert Sansom studentship.
    CG was further supported by the Institute for Basic Science (IBS-R029-C1).
    NM was supported by the German Research Foundation (DFG) with grant agreement No.\ 444419611.
    MP and WP were supported by the project BOBR that has received funding from ERC under the European Union’s Horizon 2020 research and innovation programme, grant agreement No.\ 948057.
    SzT received funding from ERC grant  BUKA (No.\ 101126229).
}%optional, to capture a funding statement, which applies to all authors. Please enter author specific funding statements as fifth argument of the \author macro.

\nolinenumbers %uncomment to disable line numbering

\EventEditors{Keren Censor-Hillel, Fabrizio Grandoni, Joel Ouaknine, and Gabriele Puppis}
\EventNoEds{4}
\EventLongTitle{52nd International Colloquium on Automata, Languages, and Programming (ICALP 2025)}
\EventShortTitle{ICALP 2025}
\EventAcronym{ICALP}
\EventYear{2025}
\EventDate{July 8--11, 2025}
\EventLocation{Aarhus, Denmark}
\EventLogo{}
\SeriesVolume{334}
\ArticleNo{166}
\begin{document}

\maketitle

\input{abstract}

\input{intro}

\input{prelims}

\input{metrics}

\input{residuality}

%%
%% Bibliography
%%

%% Please use bibtex, 

\bibliography{ref}

\appendix

\end{document}

%% file: imports.tex
\usepackage{xspace}
\usepackage{mathtools}
\usepackage{tikz}

\newtheorem{fact}[theorem]{Fact}

\newcommand{\Oof}{{\cal O}}
\DeclareMathOperator{\VCdim}{VCdim}

\renewcommand{\cal}{\mathcal}
\newcommand{\eps}{\varepsilon}

\newcommand{\FO}{\mathsf{FO}}

\newcommand{\R}{\mathbb{R}}

\renewcommand{\int}{\mathrm{int}}

\newcommand{\qr}{\mathrm{qr}}

\newcommand{\XX}{\mathcal{X}}
\newcommand{\YY}{\mathcal{Y}}
\renewcommand{\SS}{\mathcal{S}}

\newcommand{\NN}[0]{\mathrm{\mathbb{N}}}
\newcommand{\N}[0]{\mathrm{\mathbb{N}}}

\renewcommand{\subset}{\subseteq}

\newcommand{\col}{\textnormal{col}}

\newcommand{\dist}{\textnormal{dist}}

\renewcommand{\phi}{\varphi}
\newcommand{\diam}{\textnormal{diam}}

\newcommand{\floor}[1]{\lfloor #1 \rfloor}
\newcommand{\ceil}[1]{\lceil #1 \rceil}

\def\xInd#1#2{#1\setbox0=\hbox{$#1x$}\kern\wd0\hbox to 0pt{\hss$#1\mid$\hss}
	\lower.9\ht0\hbox to 0pt{\hss$#1\smile$\hss}\kern\wd0}

\def\xnotind#1#2{#1\setbox0=\hbox{$#1x$}\kern\wd0
	\hbox to 0pt{\mathchardef\nn=12854\hss$#1\nn$\kern1.4\wd0\hss}
	\hbox to 0pt{\hss$#1\mid$\hss}\lower.9\ht0 \hbox to 0pt{\hss$#1\smile$\hss}\kern\wd0}

\newcommand{\weight}{\mathbf{w}}

\newcommand{\setof}[2]{\{#1\mid#2\}}

\newcommand{\CC}{\mathcal{C}}
\newcommand{\Cc}{\CC}
\newcommand{\C}{\CC}

\newcommand{\FF}{\mathcal{F}}
\newcommand{\PP}{\mathcal{P}}

\renewcommand{\le}{\leqslant}
\renewcommand{\leq}{\le}
\renewcommand{\ge}{\geqslant}
\renewcommand{\geq}{\ge}

\newcommand{\Ball}{\mathrm{Ball}}

\newcommand{\Oh}{\mathcal{O}}

%% file: abstract.tex
\begin{abstract}
 A graph class $\CC$ is monadically dependent if one cannot interpret all graphs in colored graphs from $\CC$ using a fixed first-order interpretation. We prove that monadically dependent classes can be exactly characterized by the following property, which we call {\em{flip-separability}}: for every $r\in \N$, $\eps>0$, and every graph $G\in \Cc$ equipped with a weight function on vertices, one can apply a bounded (in terms of $\Cc,r,\eps$) number of {\em{flips}} (complementations of the adjacency relation on a subset of vertices) to $G$ so that in the resulting graph, every radius-$r$ ball contains at most an $\eps$-fraction of the total weight. On the way to this result, we introduce a robust toolbox for working with various notions of local separations in monadically dependent classes.
\end{abstract}

%% file: intro.tex
%TODO:
%\begin{enumerate}
  %\item write preliminaries
  %\item organize metric conversion section
  % \item add/prove statements about efficiency of conversions
  %\item integerate and organize proof of main result
  %\item decide which results we want to keep (variants with $r=\infty$/isolation, modeling limits, proofs of special cases)
  % \item organize appendices, remove those which are no longer needed
  % \item (uniformize notation + proof-read) x 5
  %\item write outline (at end of intro)
  %\item fit into page limit (15 pages in any format, suggested LIPICS)
  %\item hide comments
 %\item anonymize
 %\item check bibliography
  %\item check spelling
  %\item check layout
  %\item submit
%\end{enumerate}

\section{Introduction}\label{sec:intro}

In this work we study separability properties in well-structured dense graphs. To put our work in context, let us first review the setting of sparse graphs.
An archetypal statement concerning separability is the following: Every $n$-vertex tree has a {\em{centroid}} --- a vertex whose removal breaks the tree into subtrees with at most $n/2$ vertices each. This statement can be generalized to graphs of bounded treewidth: If an $n$-vertex graph $G$ has treewidth $k$, then there is a set $S$ of at most $k+1$ vertices so that every connected component of $G-S$ contains at most $n/2$ vertices. A set $S$ with this property is called a {\em{balanced vertex separator}} of $G$.

One way to generalize this statement to graphs that are not necessarily tree-like is to allow separators of larger cardinality. For instance, a classic result of Lipton and Tarjan~\cite{LiptonT80} states that every $n$-vertex planar graph has a balanced vertex separator of size $\Oh(\sqrt{n})$. Such {\em{sublinear separators}} are a fundamental tool in the algorithmic theory of graphs with topological structure. In this work, we are interested in a different kind of separators, where we still require the separator to be of bounded size, but we relax the separation requirement to have a local, rather than a global character. The statement below represents the kind of results we are interested in.

\begin{theorem}[{\cite[Thm. 42]{structural-sparsity}}]\label{thm:nd-residuality}
    For every nowhere dense graph class $\CC$, $r\in\N$, and $\eps>0$ there is some $k\in\N$ with the following property.
    For every  $n$-vertex graph $G\in \CC$
  there is a set $S$ consisting of at most $k$ vertices of $G$
  such that every ball of radius $r$ in $G-S$ contains at most $\eps\cdot n$ vertices.
\end{theorem}

\emph{Nowhere denseness} is a very general notion of uniform, local sparseness~\cite{NesetrilM11a}.
Nowhere dense classes include for example the class of planar graphs, all classes of bounded treewidth, and all classes excluding a minor; see~\cite{sparsity-book} for an introduction to the topic. Intuitively, the separator $S$ provided by \cref{thm:nd-residuality} shatters the graph's metric into small local ``vicinities'', but does not need to break the graph into components in any global sense. This way, even if a graph does not admit any sublinear-size balanced separators in the global sense, it may have a very small separator in the local sense; consider for instance subcubic expanders.

What would be an analogue of \cref{thm:nd-residuality} in dense graphs? A recent line of work has identified the model-theoretic notion of {\em{monadic dependence}} as a promising candidate for the dense counterpart of nowhere denseness. Without going into technical details, a graph class $\Cc$ is {\em{monadically dependent}} if there is no fixed one-dimensional first-order interpretation that allows interpreting all graphs in vertex-colored graphs from $\Cc$; see \cref{lem:vc-trans-mondep} for a full definition.
Apart from all nowhere dense classes, monadically dependent classes also include all monadically stable classes and all classes of bounded clique- or twin-width.
It turns out that multiple characterizations of nowhere dense classes can be lifted to analogous characterizations of monadically dependent classes (and related concepts), giving suitable dense counterparts; see e.g.~\cite{flip-flatness,flip-breakability,flipper-game,flip-width} and an overview in a recent survey~\cite{Pilipczuk25} and in the thesis~\cite{maehlmann-thesis}. In this analogy, the concept of vertex deletion is replaced by the concept of applying a {\em{flip operation}}: replacing all edges with non-edges and vice versa within some subset of vertices. Note that a single flip operation can destroy multiple edges --- for instance turn a complete graph into an edgeless graph --- hence this concept is well-suited to serve as a notion of separation in the setting of dense graphs.

And indeed, it is not hard to prove that if an $n$-vertex graph $G$ has cliquewidth at most~$k$ (cliquewidth is the dense counterpart of treewidth), then $\Oh(k)$ flip operations can be applied to $G$ so that every component of the obtained graph has at most $n/2$ vertices. This is the dense counterpart of the aforementioned separability property of graphs of bounded treewidth. The main result of this work is the dense counterpart of \cref{thm:nd-residuality}. We phrase it and prove it in the more general setting of vertex-weighted graphs. Here, a {\em{weighted graph}} is a graph $G$ equipped with a weight function $\weight\colon V(G)\to \R_{\geq 0}$; for $X\subseteq V(G)$, we denote $\weight(X)\coloneqq \sum_{u\in X} \weight(u)$. 
Additionally, $\Ball^r_G(v)$ denotes the set of vertices at distance at most~$r$ from the vertex $v$ in the graph $G$.
%All the previously mentioned separator results hold also in the weighted setting.

\begin{definition}\label{def:separable}
  A graph class $\CC$
  is \emph{flip-separable}
  if for every $r\in\N$ and $\eps>0$, there exists $k\in\N$ with the following property.
    For every graph $G\in \CC$ and weight function $\weight\colon V(G) \to \R_{\geq 0}$, there is a~graph $G'$ that can be obtained by applying at most $k$ flip operations to $G$ so that
    \[\weight(\Ball^r_{G'}(v))\le \eps\cdot \weight(V(G))\qquad\text{for every $v\in V(G)$ with weight at~most~$\eps\cdot \weight(V(G))$}.\]
\end{definition}

\begin{theorem}\label{thm:equivalence}
  A graph class $\Cc$ is monadically dependent if and only if it is flip-separable.
\end{theorem}

Together with the characterization of monadic dependence through {\em{flip-breakability}} (see \cref{thm:flip-breakability}), proposed by Dreier, M\"ahlmann, and Toru\'nczyk~\cite{flip-breakability}, \cref{thm:equivalence} corroborates the intuition that graphs from monadically dependent classes can be sparsified on a local level using few flips. In fact, we use flip-breakability in our proof, and flip-breakability can be easily derived from flip-separability (\cref{lem:sep-implies-break}); so \cref{thm:equivalence} can be seen as a strengthening of the result of~\cite{flip-breakability}. In terms of applications, we believe that \cref{thm:equivalence} will have direct consequences for the existence of {\em{modelling $\FO$-limits}} for $\FO$-convergent sequences of graphs from monadically dependent classes, similarly as is the case for \cref{thm:nd-residuality} in the context of nowhere dense classes~\cite{modellinglimits}. We defer working out this application to future work.

On our way to the proof of \cref{thm:equivalence}, we develop a versatile toolbox of {\em{flip-metrics}}: working with the local metric structure of a graph under various notions of flips.
This toolbox is an equally important contribution of this work, and we now discuss it in more detail.

\subsection*{Flip metrics and metric conversion}
There are several, closely related variants of flips, and of the \emph{metrics} they define, each of which has its advantage over the others. The emerging toolbox  allows converting one variant into another, taking advantage of the benefits of each variant.

\subparagraph*{Partition flips.} So far, we have only discussed the ``operational'' concept of flips, which boils down to applying a number of flip operations --- complementing the adjacency relation within some set of vertices. It is equivalent, but more useful, to see this as a single operation consisting of taking a partition of the vertex set and complementing the adjacency relation within a selection of pairs of parts. Concretely, if $G$ is a graph and $\cal P$ is a partition of vertices of $G$, then a {\em{$\cal P$-flip}} of $G$ is any graph $G'$ that can be obtained from $G$ as follows: for every pair of parts $A,B\in \cal P$ (possibly $A=B$), either complement the adjacency relation within $A\times B$ or leave it intact. Further, call $G'$ a {\em{$k$-flip}} of $G$ if $G'$ is a $\cal P$-flip of $G$ for some partition $\cal P$ with $|\cal P|\leq k$. It can be easily seen that if $G'$ is a $k$-flip of $G$, then $G'$ can be obtained from $G$ by applying $\Oh(k^2)$ single flip operations; and if $G'$ can be obtained from $G$ by applying $\ell$ single flip operations, then $G'$ is a $2^\ell$-flip of $G$. Hence, from now on we will use the definition of a $k$-flip of a graph as our basic notion of flips. In particular, in \cref{def:separable} we can equivalently postulate that $G'$ is a $k$-flip of $G$.

\subparagraph*{Definable flips.} The caveat of $k$-flips is that the definition considers an arbitrary partition $\cal P$ of the vertex set. In many arguments, particularly those concerning first-order logic, it is useful to restrict attention to some well-behaved partitions, for instance definable from a handful of vertices. For this purpose, one considers {\em{definable flips}}, first used by Bonnet et al.~\cite{boundedLocalCliquewidth}, and then more explicitly by Gajarsk\'y et al.~\cite{flipper-game}. 
Concretely, if $S$ is a set of vertices of a graph~$G$, then we consider the partition $\cal P_S$ in which every $s\in S$ is in its own singleton part, while the vertices $v\in V(G)\setminus S$ are partitioned according to their neighborhood in $S$, i.e.~$u,v$ are in the same part if $N_G(u) \cap S = N_G(v) \cap S$, where $N_G(u)$ denotes the (open) neighborhood of $u$ in $G$. Then, $\PP_S$-flips are called \emph{$S$-definable flips}.

While definable flips are in general less powerful than classic flips, it turns out that in classes of bounded VC-dimension, every classic flip can be in some sense emulated by a definable flip. This observation is formalized in the lemma below, first proved by Bonnet et al.~\cite{boundedLocalCliquewidth} (we use the formulation of Toru\'nczyk~\cite{flip-width}). Recall that the \emph{VC-dimension} of a graph $G$ is the maximum cardinality of a set $A\subset V(G)$ such that $\setof{N_G(v)\cap A}{v\in V(G)}=2^A$; monadically dependent classes have bounded~VC-dimension.

\begin{restatable}[\cite{boundedLocalCliquewidth,flip-width}]{lemma}{inclem}\label{lem:inc-lemma}
  Let $G$ be a graph of VC-dimension at most $d$ and let $G'$ be a $k$-flip of $G$.
  Then there is a~set $S\subset V(G)$ with $|S|\le \Oh(dk^2)$ and an $S$-definable flip $G''$ of $G$
  so that for all $r\in\N$, we have
  $$\Ball^r_{G''}(v)\subset \Ball^{5r}_{G'}(v)\qquad\text{for all $v\in V(G)$.}$$
\end{restatable}

The right way of seeing the conclusion of \cref{lem:inc-lemma} is that $G''$ sparsifies $G$ at least as well as $G'$, because bounded-radius balls in $G''$ are contained in bounded-radius balls in $G'$.
\Cref{lem:inc-lemma} turned out to be an indispensable tool in the study of flips, see its applications in \cite{boundedLocalCliquewidth,definable-decompositions,flip-width}.%\mipi{Nothing more?}

\subparagraph*{Flip-metrics.} A fundamental property of vertex separators is that they are easily aggregable. For instance, let $G$ be a graph with some vertices colored red and some colored blue, and let $S_1,S_2$ be vertex sets such that every component of $G-S_1$ contains at most $p$ red vertices, and every component of $G-S_2$ contains at most $p$ blue vertices. Then, the union $S_1\cup S_2$ is a separator achieving both properties: every component of $G-(S_1\cup S_2)$ has at most $p$ red vertices {\em{and}} at most $p$ blue vertices. However, it is not obvious how to achieve this kind of aggregability in the context of flips: if $G_1,G_2$ are $k$-flips of $G$ so that  every component of $G_1$ has at most $p$ red vertices and every component of $G_2$ has at most $p$ blue vertices, how do we construct a single flip $G'$ that satisfies both these properties?

A way of approaching this issue was proposed by Gajarsk\'y et al.~\cite{flipper-game} in the form of {\em{flip-metrics}}. In a nutshell, the idea is that given a set of vertices $S$, we consider all the $S$-definable flips at the same time. Concretely,  for two vertices $u,v$, we define
\[\dist_S(u,v)\coloneqq \max\{\dist_{G'}(u,v)\colon G'\textrm{ is an $S$-definable flip of $G$}\}.\]
As observed in~\cite{flipper-game}, this is a metric on the vertex set of $G$. Note that balls in this metric satisfy the following:
\[\Ball_S^r(v)
\coloneqq \{ u \in V(G) : \dist_S(u,v) \leq r \}
=\bigcap \left\{\Ball_{G'}^r(v)\colon G'\textrm{ is an $S$-definable flip of $G$}\right\}.\]
Crucially, the flip-metrics are aggregable in the following sense: for two vertex sets $S,T$ we always have $\dist_{S\cup T}(u,v)\geq \max(\dist_S(u,v),\dist_T(u,v))$, hence every $r$-ball in the metric $\dist_{S\cup T}$ is contained in the intersection of $r$-balls in the metrics $\dist_S$ and $\dist_T$.

The caveat of flip-metrics is that they no longer originate from a single graph on which one could work. Our main contribution to the theory of flips is the following lemma, which shows that in classes of bounded VC dimension, any flip-metric can be emulated by the usual distance metric in a single definable~flip.

\begin{restatable}{lemma}{conversion}\label{lem:new-transfer}
  Let $G$ be a graph of VC-dimension at most $d$ and let $T$ be a set of vertices of $G$ of size at most $k$. Then there is a set of vertices $S$ with $|S|\leq k^{\Oh(d^2)}$ and an $S$-definable flip $G'$ of $G$ such that
    \[\Ball^r_{G'}(v)\subset \Ball^{30r}_T(v)\qquad\text{for all $v\in V(G)$ and $r \in \N$.}\]
\end{restatable}

%Again, the right way to think about it is that the metric $\dist_{G'}$ can be only sparser than the metric $\dist_S$, because bounded-radius balls in $\dist_{G'}$ are contained in bounded-radius balls in $\dist_S$.

\cref{lem:inc-lemma,lem:new-transfer} show that in graph classes of bounded VC-dimension, all the discussed viewpoints --- in the paragraphs \emph{partition flips}, \emph{definable flips}, and \emph{flip-metrics} --- are essentially equivalent, and one can easily move from one viewpoint to the other depending on specific properties that are useful in a context. We showcase this in the proof of \cref{thm:equivalence}, given is \cref{sec:flip-separability}. However, we expect that the combination of \Cref{lem:inc-lemma,lem:new-transfer} will have further applications in the treatment of monadically dependent classes.

\subsection*{Outline of the proof of \cref{thm:equivalence}.}
Let us close this introduction by sketching the proof of our main result.
We will present an incorrect proof attempt, and explain the changes needed to fix it.

We are given~$G$ in the monadically dependent class~$\Cc$ with weights~$\weight$, and search for a $k$-flip of~$G$ in which all $r$-balls have at~most an $\eps$-proportion of the total weight.
Using the metric conversion results, it is enough to find a bounded set~$S$ for which the above holds in the flip metric~$\dist_S$; we say that such an~$S$ \emph{sparsifies}~$G$.
We will start with $S = V(G)$ which trivially sparsifies~$G$, and show that as long as~$S$ is very large, it can be modified to reduce its size.

We first apply flip-breakability~\cite{flip-breakability} to~$S$:
this states that in the very large set~$S$, we can find large subsets~$A_1,A_2$ with $\dist_F(A_1,A_2) > 2r'$ for some bounded flip~$F$, and some fixed distance~$r'$ we will choose later.
By repeated application, we can instead find~$p \coloneqq \ceil{\frac{1}{\eps}}$ (which is a constant) large subsets $A_1,\dots,A_p$ pairwise at distance~$2r'$.
Then the $r'$-balls around all the~$A_i$s are pairwise disjoint, and one of them, say~$\Ball_F^r(A_i)$, must have weight at most~$\eps \cdot \weight(V(G))$.

We now want to modify~$S$ by adding~$F$ and removing all but a bounded subset~$X$ of~$A_i$.
Since~$F$ is bounded and~$A_i$ is large, this will decrease the size of~$S$.
To show that the resulting set $S' = S - A_i + X + F$ still sparsifies~$G$, we have to consider two cases:
\begin{enumerate}
  \item If~$v$ is close to~$A_i$, meaning $\dist_F(v,A_i) \le r'-r$, then~$F$ already is enough to sparsify the ball around~$v$, as it is contained in~$\Ball_F^r(A_i)$.
  \item On the other hand, if~$v$ at distance more than~$r'-r$ from~$A_i$ for an appropriately large choice of~$r'$,
    then we would like to use a locality result to show that most of~$A_i$ is irrelevant in sparsifying the ball around~$v$:
    we can find a bounded set~$X$ which emulates the behaviour of all vertices in~$A_i$ and whose choice does not depend on~$v$, allowing to delete~$A_i-X$ from~$S$.
\end{enumerate}
The issue is of course that this supposed locality result needed in the second case of the proof fails.
We are considering all $S$-flips for this unbounded set~$S$, and this appears to be far too powerful for this kind of arguments to hold.

A similar and correct locality statement goes informally as follows:
in a large collection~$\mathcal{A}$ of sets of size~$t$, there is some bounded subcollection~$\XX \subseteq \mathcal{A}$ such that
if~$u,v$ are far from all of~$\mathcal{A}$, and~$\dist_S(u,v) > r$ holds for some~$S \in \mathcal{A}$, then it also holds for some~$S \in \XX$.
This is a simple corollary of Gaifman's locality for first-order logic.
Note that while~$\mathcal{A}$ is unbounded, we only consider $S$-flips for sets~$S \in \mathcal{A}$ of bounded size~$t$.

This suggests a way to fix the proof:
rather than having one large set~$S$ and considering \mbox{$S$-flips}, we consider a large family~$\FF$ of sets of size~$t$, and work with the metric $\dist_\FF(u,v) = \max_{S \in \FF} \dist_S(u,v)$ combining $S$-flips for all~$S \in \FF$.
The previous sketch can be adjusted to this new setting without major changes if one picks~$t$ to be the size of flips produced by flip-breakability, and once we reach the last case of the proof, the form of locality required is exactly the one stated above.

%% file: prelims.tex
\section{Preliminaries}\label{sec:prelims}

For a nonnegative integer $p$, we denote $[p]\coloneqq\{1, \ldots, p\}$.
If $\mathcal X$ is a~set of sets, then $\bigcup \mathcal X$ is their union.

\subsection{Standard graph-theoretic notation}

We denote by $V(G)$ and $E(G)$ the sets of vertices and of edges of a graph $G$, respectively.
A~graph $H$ is a~\emph{subgraph} of a~graph $G$ if $H$ can be obtained from $G$ by vertex and edge deletions.
Graph~$H$ is an~\emph{induced subgraph} of $G$ if $H$ is obtained from $G$ by vertex deletions only.
For $S \subseteq V(G)$, the \emph{subgraph of $G$ induced by $S$}, denoted $G[S]$, is obtained by removing from $G$ all the vertices that are not in $S$ (together with their incident edges).
Then, $G-S$ is a shorthand for $G[V(G)\setminus S]$.

We denote the (open) neighborhood of a vertex $v$ in $G$ by $N_G(v)$ and by $\Ball_G^r(v)$ the set of vertices at distance at~most~$r$ from $v$ in~$G$. We set $\Ball_G^r(S) \coloneqq \bigcup_{v \in S} \Ball_G^r(v)$ for $S \subseteq V(G)$.
The \emph{diameter} of~a~graph~$G$, denoted by $\diam(G)$, is defined as $\max_{u, v \in V(G)} \dist_G(u,v)$ where $\dist_G(u,v)$ is the shortest-path distance between $u$ and $v$.
A~\emph{biclique} in $G$ is made of two non-empty disjoint sets $A, B \subset V(G)$ such that for every $a \in A$ and for every $b \in B$, $ab \in E(G)$.
For any $A, B \subset V(G)$, we denote by $E_G(A,B)$ the edge set $\{uv \in E(G)~|~u\in A, v\in B\}$.
If furthermore, $A$ and $B$ are disjoint, we denote by $G[A,B]$ the bipartite subgraph of $G$ with bipartition $(A,B)$ and edge set $E_G(A,B)$.

\subsection{Flips}\label{sec:flip-def}

Given a~graph $G$ and two not necessarily disjoint subsets $A, B \subseteq V(G)$, the \emph{$(A,B)$-flip} of $G$ is the graph with vertex set $V(G)$ and edge set \[E(G)\ \triangle\  \{ab~: a \in A,\ b \in B,\ a \neq b\},\]
where $\triangle$ is the symmetric difference.
A~\emph{flip} of a~graph $G$ is any \emph{$(A,B)$-flip} of $G$ for some $A, B \subseteq V(G)$.

Given a~partition $\cal P$ of~$V(G)$, a~\emph{$\cal P$-flip} of a~graph~$G$ is any graph obtained from $G$ by performing a~sequence of flips, each of which is a~$(P,P')$-flip for some (possibly equal) $P, P' \in \cal P$.
Note that a~\emph{$\cal P$-flip} of~$G$ is fully determined by specifying a~graph on vertex set $\cal P$ with possible loops.
In particular, there are at~most~$2^{|\cal P|^2}$ many $\cal P$-flips of~$G$, and the sequence of flips can always be chosen to be of length at~most~$|\cal P|^2$.
Then, for any positive integer $k$, a~\emph{$k$-flip} of $G$ is any $\cal P$-flip of $G$ with $|\cal P|\leq k$, i.e., $\cal P$ comprises at most $k$ parts.

For any $S \subseteq V(G)$, the \emph{partition of $V(G)$ defined by $S$}, denoted by $\cal P_S$, is \[\{\{s\}~\colon~s \in S\} \cup \{\{v \in V(G-S)~|~N_G(v) \cap S=S'\}~\colon~S' \subseteq S\},\]
where empty sets are removed from the partition.
Note that $\cal P_S$ is indeed a~partition of~$V(G)$, with at~most $|S|+2^{|S|}$ parts.
An \emph{$S$-definable flip} of~$G$ is simply a~$\cal P_S$-flip of~$G$, and a~\emph{$k$-definable flip} is an~$S$-definable flip for some $S \subseteq V(G)$ of size at~most~$k$.

We associate distances to flips.
Recall that, given $u,v \in V(G)$, $\dist_G(u,v)$ denotes the \emph{shortest-path distance} between $u$ and $v$ in $G$.
Then, for any $\cal P$ partition of $V(G)$, we define
\[\dist^G_{\cal P}(u,v) \coloneqq \max\,\{\,\dist_{G'}(u,v)\colon G'\text{ is a }\cal P\textrm{-flip of }G\,\};\]
and for $S \subseteq V(G)$,
\[\dist^G_S(u,v) \coloneqq \max\,\{\,\dist_{G'}(u,v)\colon G'\text{ is an }S\textrm{-definable flip of }G\,\}.\]
We also define the corresponding balls
\[
  \Ball_\PP^r(v) \coloneqq \{ u \in V(G) : \dist^G_\PP(u,v) \leq r \} 
  \quad
  \text{and}
  \quad
  \Ball_S^r(v) \coloneqq \{ u \in V(G) : \dist^G_S(u,v) \leq r \}, 
\]
where the graph $G$ will always be clear from the context.
In these cases we may also omit $G$ from sub- or superscripts.

\subsection{VC-dimension, transductions, and monadic dependence}\label{lem:vc-trans-mondep}

We recall the standard terminology of VC-theory in the context of graphs.
Given a graph $G$ and a set of vertices $A$, we say that  $G$ \emph{shatters} $A$ if for each $B\subseteq A$ there exists a vertex $u_B$ such that $N(u_B)\cap A=B$.
We define the \emph{VC-dimension} of $G$, denoted as $\VCdim(G)$, as the maximum size of a set $A\subseteq V(G)$ shattered by $G$. For a graph class $\C$ we write $\VCdim(\C)\coloneqq\sup\,\{\,\VCdim(G)\colon G \in \C\,\}$; note that it may happen that $\VCdim(\C)=\infty$ in case graphs from $\C$ contain arbitrarily large shattered sets.
We additionally define the \emph{growth function}, or \emph{shatter function}, of $G$ as the map $\pi_G:\N \to \N$ given by 
\[ \pi_G(n)\coloneqq \max_{A \subseteq V(G), |A|=n}\,|\{\, N_G(v)\cap A\,\colon\,v \in V(G)\,\}|. \]
Evidently $\pi_G(n)\leq 2^n$, while $\VCdim(G)<n \iff \pi_G(m) < 2^m$ for every $m \geq n$. However, a bound on the VC-dimension implies a much stronger bound on the shatter function.

\begin{fact}[Sauer-Shelah Lemma, \cite{Sauer72,Shelah72}]\label{sauershelah}
  Let $G$ be a graph with $k \coloneqq \VCdim(G)$.
  Then for all $n\geq k$, we have
    \[ \pi_G(n) \leq \sum_{i=0}^k \binom{n}{i}.\]
    In particular, $\pi_G(n) \leq {\cal O}(n^k)$. 
\end{fact}

We now recall the definition of \emph{first-order transductions}, or \emph{transductions} for short, in the special case of languages over (vertex-colored) graphs; see also the discussion in \cite{Pilipczuk25} for a broader introduction.
A~(vertex-colored) graph is here seen as a relational structure consisting of the vertex set equipped with one binary relation $E(\cdot, \cdot)$ signifying adjacency, and one unary relation per color.
Given a~non-negative integer~$k$, a~transduction $\mathsf T$ is specified by a set of colors $\Sigma$ and a~first-order formula $\varphi(x,y)$ with two free variables in the language of $\Sigma$-colored graphs.
For every (uncolored) graph $G$, we denote by $\mathsf T(G)$ the family of all induced subgraphs of graphs $H$ with $V(H)=V(G)$ satisfying the following: there are subsets $\{U_C\colon C\in \Sigma\}$ of $V(G)$ such that for all distinct $u,v\in V(H)$, we have $uv\in E(H)$ if and only if $\varphi(u,v)$ holds in $G$ with vertices of $U_C$ marked as of color $C$, for each $C\in \Sigma$.

Given a~graph class~$\mathcal C$, we denote by $\mathsf T(\mathcal C)$ the class $\bigcup_{G \in \mathcal C}\mathsf T(G)$.
We say that a~graph class $\mathcal C$ \emph{transduces} a~class $\mathcal D$ if there is a~transduction $T$ such that $\mathcal D \subseteq \mathsf T(\mathcal C)$.
Finally, a~characterization of Baldwin and Shelah~\cite{baldwin-shelah} states that a~graph class is \emph{monadically dependent} (or \emph{monadically NIP}) if it does not transduce the class of all graphs.
We take it as our definition of \emph{monadic dependence}.

The \emph{quantifier rank} of a~(first-order) formula $\varphi$, denoted by $\qr(\varphi)$, is the maximum number of nested quantifiers in $\varphi$.

%% file: metrics.tex
\label{sec:metric-conversion}
\section{Metric conversion}

% the diameter bounds in case we need to update them again
\newcommand{\diamBound}{3\xspace}
\newcommand{\diamBoundBip}{6\xspace}

In this section we prove \Cref{lem:new-transfer}. In fact, we shall establish the following stronger result that more generally shows that the metric arising from an arbitrary partition can be approximated by that coming from a definable flip.

\begin{theorem}\label{thm:conversion}
  For every graph $G$ of VC-dimension $d$ and a partition $\PP$ of $V(G)$, there is a set of vertices $S$ of size $\cal{O}(d \cdot |\PP|^{2d+2})$ and an $S$-definable flip $G'$ of $G$ such that for every $r\in\N$, $v \in V(G)$, and $\PP$-flip $H$~of~$G$
  \[
      \Ball^r_{G'}(v) \subseteq 
      \Ball^{30r}_{H}(v).
  \]
\end{theorem}

Note that \Cref{lem:new-transfer} follows immediately from \cref{thm:conversion}. This is because for any vertex set $T$ of size~$k$, the partition $\PP_T$ has size $\Oh(k^d)$ where $d$ is the VC-dimension, by the Sauer-Shelah Lemma (\Cref{sauershelah}); and the definition of the metric $\dist_T$ considers all $\PP_T$-flips.

Towards \Cref{thm:conversion}, we argue that the partition metric can be approximated by the distance metric of a concrete, not necessarily definable, flip. 

\begin{lemma}
  Let~$G$ be a graph and~$\PP$ a partition of the vertex set of $G$.
  Then there exists a refinement~$\PP'$ of~$\PP$ and a $\PP'$-flip~$G'$ of $G$ such that for all~$v \in V(G)$,
  \[ \Ball^r_{G'}(v) \subseteq \Ball^{6r}_\PP(v). \]
  Furthermore, we have  $|\PP'| \le |\PP| \cdot2^{|\PP|}$, and if~$G$ has VC-dimension at most~$d$, then $|\PP'| = \Oof(|\PP|^{d+1})$.
  Moreover, $G'$ and $\PP'$ can be computed from $G$ and $\PP$ in time $\cal{O}(|\PP|^2\cdot |V(G)|^2)$.
  \label{lem:informal}
\end{lemma}

Evidently, this result together with \Cref{lem:inc-lemma}, which allows to approximate the metric of an arbitrary flip by that of a definable flip, imply \Cref{thm:conversion}. We now turn to the proof of \Cref{lem:informal}. For this, we first use some folklore results.
Let~$\overline{G}$ denote the complement of a graph~$G$, i.e., the graph obtained from $G$ by replacing all edges with non-edges and vice versa.
\begin{lemma}
  \label{lem:diam-complement}
  For any graph~$G$, we have~$\diam(G) \le \diamBound$ or~$\diam(\overline{G}) \le \diamBound$.
\end{lemma}
\begin{proof}
  Assume that~$\diam(G) > 3$, i.e., there are vertices~$u,v$ with~$\dist_G(u,v) > 3$ (possibly~$u,v$ are in distinct connected components).
  Then there is no common neighbor of~$u,v$.
  Thus, in the complement~$\overline{G}$, all vertices are neighbors of either~$u$ or~$v$, and~$uv$ is an edge.
  This implies that~$\diam(\overline{G}) \le 3$.
\end{proof}
Observe that \cref{lem:diam-complement} proves \cref{lem:informal} in the case where the partition~$\PP = \{V(G)\}$ is trivial. Indeed, suppose first that $\diam(G)\leq \diamBound$. We then let $\PP'=\PP$ and $G'=\overline{G}$, which is a~$\PP$ flip of~$G$. Thus, for every edge~$uv \in E(G')$, we have that
the edge~$uv$ is a path of length~$1$ connecting $u$ and $v$ in~$\overline{G}$, and there is always a~path connecting $u$ and $v$ of length~\diamBound in $G$. Consequently $\dist_{\PP}(u,v)\leq \diamBound$ and, more generally, $\dist_{\PP}(u,v) \leq \diamBound \cdot \dist_{G'}(u,v)$. The case when $\diam(\overline{G})\leq \diamBound$ can be handled by a symmetric argument.

To generalize this, we need a variant of \cref{lem:diam-complement} for bipartite graphs.
Let $B = (U,V,E)$ be a bipartite graph, where~$U,V$ are the two sides of the bipartition.
We consider the choice of bipartition (which is not necessarily unique) to be a part of~$B$.
The bipartite complement~$\overline{B}$ of~$B$ is the bipartite graph with edge set $\{uv : \text{$u \in U$, $v \in V$ and $uv \not\in E(B)$} \}$.
A degenerate case appears in this bipartite setting:
it is possible for both~$B$ and~$\overline{B}$ to be disconnected (which is impossible for the usual notion of complement).
\begin{lemma}
  \label{lem:diam-complement-bipartite}
  For any bipartite graph~$B$,
  either~$\diam(B) \le \diamBoundBip$, $\diam(\overline{B}) \le \diamBoundBip$, or both~$B$ and~$\overline{B}$ are disconnected.
\end{lemma}

\begin{proof} See \cref{fig:bicomplement-diam} for an illustration.
  Assume that one of~$B,\overline{B}$ is connected, say~$B$.
  Suppose furthermore that~$\diam(B) > 6$. We must then show~$\diam(\overline{B}) \le 6$.
  Using that~$B$ is connected, we can pick vertices~$x,y$ at distance exactly~6,
  and a shortest path $x = x_0, \dots, x_6 = y$ between them.
  This path alternates between the sides~$U,V$ of the bipartition, say $x_i \in U$ for even~$i$ and $x_i \in V$ for odd~$i$.

  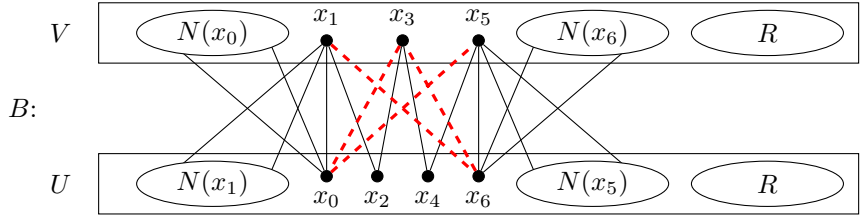
\begin{figure}[ht]
    \centering
    \begin{tikzpicture}
      \tikzstyle{vertex}=[black,fill,draw,circle,inner sep=0pt, minimum width=1.5mm]
      \node at (-4,1) {$B$:};
      % bipartition
      \draw (-3,-0.4) rectangle (7,0.4);
      \node at (-3.5,0) {$U$};
      \draw (-3,1.6) rectangle (7,2.4);
      \node at (-3.5,2) {$V$};

      % path
      \foreach \i in {0,2,4,6}{
        \node[vertex, label=below:$x_\i$] (x\i) at (\i/3,0.1) {};
      }
      \foreach \i in {1,3,5}{
        \node[vertex, label=above:$x_\i$] (x\i) at (\i/2-0.5,1.9) {};
      }
      \draw (x0) -- (x1) -- (x2) -- (x3) -- (x4) -- (x5) -- (x6);
      \draw[dashed, very thick,red] (x1) -- (x6) -- (x3) -- (x0) -- (x5);

      % neighborhoods
      \foreach \i/\x/\y in {0/-1.5/2,1/-1.5/0,5/3.5/0,6/3.5/2}{
        \draw (x\i) -- (\x+0.7,\y) (x\i) -- (\x-0.7,\y);
        \draw[fill=white] (\x,\y) ellipse (1 and 0.3);
        \node at (\x,\y) {$N(x_\i)$};
      }

      \draw[fill=white] (5.8,0) ellipse (1 and 0.3);
      \draw[fill=white] (5.8,2) ellipse (1 and 0.3);
      \node at (5.8,0) {$R$};
      \node at (5.8,2) {$R$};
    \end{tikzpicture}
    \caption{%
      A bipartite graph~$B$ with a shortest path $x_1,\dots,x_6$ of length~$6$.
      Note that the neighborhoods of~$x_0,x_1,x_5,x_6$ must be disjoint.
      The set~$R$ denotes what remains outside these neighborhoods.
      The red-dashed edges form a path~$P$ of length~4 in the bipartite complement~$\overline{B}$ which dominates all of~$\overline{B}$,
      i.e., no vertex of~$B$ is adjacent to all of~$P$.
      This ensures $\diam(\overline{B}) \le 6$.
    }
    \label{fig:bicomplement-diam}
  \end{figure}

  Now a vertex~$u \in U$ cannot be adjacent to both~$x_1$ and~$x_5$, as these vertices are at distance~$4$.
  Similarly, a vertex $v \in V$ cannot be adjacent to both~$x_0$ and~$x_6$.
  It follows that in~$\overline{B}$, all vertices are adjacent to one of~$x_0,x_1,x_5,x_6$,
  and the latter are connected by the following path of length~4: $(x_5,x_0,x_3,x_6,x_1)$.
  It follows that~$\diam(\overline{B}) \le 6$.
\end{proof}

The degenerate cases of \cref{lem:diam-complement-bipartite} have a simple classification.
\begin{lemma}
  \label{lem:bipartite-disco-classification}
  Let~$B = (U,V,E)$ be a bipartite graph such that both~$B$ and~$\overline{B}$ are disconnected. Then
  \begin{enumerate}
    \item \label{it1:bdc} either~$B$ is the disjoint union of two bicliques, or
    \item \label{it2:bdc} up to swapping~$U$ and~$V$, there are vertices~$v^-,v^+ \in V$
      such that~$v^-$ is isolated and~$v^+$ is adjacent to all of~$U$.
  \end{enumerate}
\end{lemma}

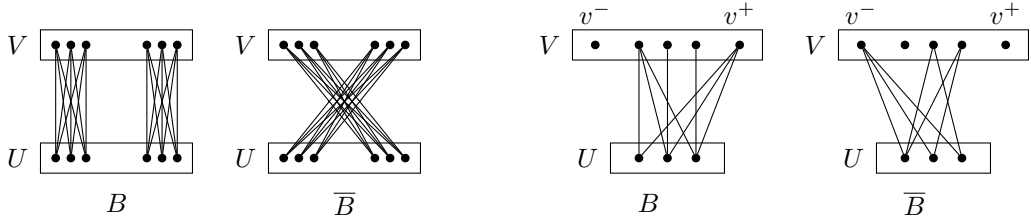
\begin{figure}[ht]
  \centering
  \begin{tikzpicture}
    \tikzstyle{vertex}=[black,fill,draw,circle,inner sep=0pt, minimum width=1mm]

    \begin{scope}
      \node at (1,-0.6) {$B$};

      \draw (0,-0.2) rectangle (2,0.2);
      \node at (-0.3,0) {$U$};
      \draw (0,1.3) rectangle (2,1.7);
      \node at (-0.3,1.5) {$V$};

      \foreach \i/\x in {1/0.2,2/0.4,3/0.6, 4/1.4,5/1.6,6/1.8}{
        \node[vertex] (u\i) at (\x,0) {};
        \node[vertex] (v\i) at (\x,1.5) {};
      }
      \foreach \i in {1,2,3}{\foreach \j in {1,2,3}{
        \draw (u\i) -- (v\j);
      }}
      \foreach \i in {4,5,6}{\foreach \j in {4,5,6}{
        \draw (u\i) -- (v\j);
      }}
    \end{scope}
    \begin{scope}[xshift=3cm]
      \node at (1,-0.6) {$\overline{B}$};

      \draw (0,-0.2) rectangle (2,0.2);
      \node at (-0.3,0) {$U$};
      \draw (0,1.3) rectangle (2,1.7);
      \node at (-0.3,1.5) {$V$};

      \foreach \i/\x in {1/0.2,2/0.4,3/0.6, 4/1.4,5/1.6,6/1.8}{
        \node[vertex] (u\i) at (\x,0) {};
        \node[vertex] (v\i) at (\x,1.5) {};
      }
      \foreach \i in {1,2,3}{\foreach \j in {4,5,6}{
        \draw (u\i) -- (v\j) (v\i) -- (u\j);
      }}
    \end{scope}

    \begin{scope}[xshift=7cm]
      \node at (1,-0.6) {$B$};

      \draw (0.5,-0.2) rectangle (2,0.2);
      \node at (0.2,0) {$U$};
      \draw (0,1.3) rectangle (2.5,1.7);
      \node at (-0.3,1.5) {$V$};

      \node[vertex, label={[label distance=1mm]above:$v^-$}] (vm) at (0.3,1.5) {};
      \node[vertex, label={[label distance=1mm]above:$v^+$}] (vp) at (2.2,1.5) {};

      \foreach \i in {1,...,3}{
        \node[vertex] (u\i) at (0.5+0.375*\i,0) {};
        \node[vertex] (v\i) at (0.5+0.375*\i,1.5) {};
        \draw (vp) -- (u\i);
      }

      \foreach \v/\u in {1/1,1/2,1/3,2/2,3/3}{
        \draw (u\u) -- (v\v);
      }
    \end{scope}

    \begin{scope}[xshift=10.5cm]
      \node at (1,-0.6) {$\overline{B}$};

      \draw (0.5,-0.2) rectangle (2,0.2);
      \node at (0.2,0) {$U$};
      \draw (0,1.3) rectangle (2.5,1.7);
      \node at (-0.3,1.5) {$V$};

      \node[vertex, label={[label distance=1mm]above:$v^-$}] (vm) at (0.3,1.5) {};
      \node[vertex, label={[label distance=1mm]above:$v^+$}] (vp) at (2.2,1.5) {};

      \foreach \i in {1,...,3}{
        \node[vertex] (u\i) at (0.5+0.375*\i,0) {};
        \node[vertex] (v\i) at (0.5+0.375*\i,1.5) {};
        \draw (vm) -- (u\i);
      }

      \foreach \v/\u in {2/1,2/3,3/1,3/2}{
        \draw (u\u) -- (v\v);
      }
    \end{scope}

  \end{tikzpicture}
  \caption{The two types of bipartite graphs~$B$ with both~$B$ and~$\overline{B}$ disconnected (\cref{lem:bipartite-disco-classification}).}
  \label{fig:bipartite-disco-classification}
\end{figure}

\begin{proof}
  Assume first that~$B$ contains an isolated vertex~$v^-$, which without loss of generality is in~$V$.
  Thus in~$\overline{B}$, $U \cup \{v^-\}$ is connected.
  If all vertices of~$V$ had a non-neighbor in~$U$, then~$\overline{B}$ would be connected, see \cref{subfig:isolated-vertex}.
  Hence, there must be some~$v^+ \in V$ connected to all of~$U$, and case~\ref{it2:bdc} of the lemma holds.

  Otherwise, every connected component of $B$ has at least two vertices. Let $C_1,\dots,C_k$ be the components of $B$, where $k\ge 2$, for we assume that $B$ is disconnected.
  Call~$U_i = U \cap C_i$ and~$V_i = V \cap C_i$ the two sides of component $C_i$; they are non-empty since $C_i$ has more than one vertex.
  Note that for all $i \neq j$, $(U_i,V_j)$ is a biclique in~$\overline{B}$.
  If there are at least~3 components $C_i$, it is then simple to check that~$\overline{B}$ is connected, a contradiction (see \cref{subfig:three-components}).

  We finally assume that there are only two components~$C_1,C_2$, each with at least two vertices.
  Once again, $(U_1,V_2)$ and~$(U_2,V_1)$ are bicliques in~$\overline{B}$.
  If in either~$C_1$ or~$C_2$ there is some non-edge between~$U$ and~$V$,
  then the corresponding edge in~$\overline{B}$ connects these two bicliques, and thus all of~$\overline{B}$, a contradiction (see \cref{subfig:non-complete-component}).
  Therefore,~$C_1$ and~$C_2$ are themselves bicliques in~$B$, so case~\ref{it1:bdc} of the lemma holds.
\end{proof}

\begin{figure}[ht]
  \centering
  \begin{subfigure}[t]{0.3\textwidth}
    \centering
    \begin{tikzpicture}
      \tikzstyle{vertex}=[black,fill,draw,circle,inner sep=0pt, minimum width=1.5mm]
        \draw (-0.2,-0.2) rectangle (3.2,0.2);
        \node at (-0.5,0) {$U$};
        \draw (-0.2,1.3) rectangle (3.2,1.7);
        \node at (-0.5,1.5) {$V$};

        \node[vertex, label={[label distance=1mm]above:$v^-$}] (vm) at (0,1.5) {};
        \foreach \i in {0,...,4}{
          \node[vertex] (u\i) at (1+0.5*\i,0) {};
          \node[vertex] (v\i) at (1+0.5*\i,1.5) {};

          \draw[dashed] (vm) -- (u\i);
        }
        \draw[dashed] (v0) -- (u1) (v1) -- (u1) (u0) -- (v2) -- (u3) (u2) -- (v3) -- (u4) (v4) -- (u4);
    \end{tikzpicture}
    \caption{Vertex~$v^- \in V$ is isolated, and no~$v \in V$ dominates~$U$, i.e., all of them have a non-edge to~$U$.}
    \label{subfig:isolated-vertex}
  \end{subfigure}
  \hspace{0.01\textwidth}
  \begin{subfigure}[t]{0.3\textwidth}
    \centering
    \begin{tikzpicture}
      \tikzstyle{vertex}=[black,fill,draw,circle,inner sep=0pt, minimum width=1mm]
        \draw (-0.2,-0.2) rectangle (3.2,0.2);
        \node at (-0.5,0) {$U$};
        \draw (-0.2,1.3) rectangle (3.2,1.7);
        \node at (-0.5,1.5) {$V$};

        % components
        \foreach \i in {1,2,3}{
          \draw (\i*1.2-0.9,0.75) ellipse (0.3 and 1);
          \node at (\i*1.2-0.9,2) {$C_\i$};
        }

        \foreach \i/\x in {0/0.2,1/0.4,2/1.4,3/1.6,4/2.6,5/2.8}{
          \node[vertex] (u\i) at (\x,0) {};
          \node[vertex] (v\i) at (\x,1.5) {};
        }

        \foreach \i/\j in {0/2,0/3,0/4,0/5,1/2,1/3,1/4,1/5,2/4,2/5,3/4,3/5}{
          \draw[dashed] (u\i) -- (v\j) (v\i) -- (u\j);
        }
    \end{tikzpicture}
    \caption{There are at least~3 non-trivial connected components $C_1,C_2,C_3$.}
    \label{subfig:three-components}
  \end{subfigure}
  \hspace{0.01\textwidth}
  \begin{subfigure}[t]{0.3\textwidth}
    \centering
    \begin{tikzpicture}
      \tikzstyle{vertex}=[black,fill,draw,circle,inner sep=0pt, minimum width=1mm]
        \draw (-0.2,-0.2) rectangle (3.2,0.2);
        \node at (-0.5,0) {$U$};
        \draw (-0.2,1.3) rectangle (3.2,1.7);
        \node at (-0.5,1.5) {$V$};

        % components
        \foreach \i in {1,2}{
          \draw (\i*1.8-1.2,0.75) ellipse (0.7 and 1);
          \node at (\i*1.8-1.2,2) {$C_\i$};
        }

        \foreach \i/\x in {1/0.3,2/0.6,3/0.9, 4/2.1,5/2.4,6/2.7}{
          \node[vertex] (u\i) at (\x,0) {};
          \node[vertex] (v\i) at (\x,1.5) {};
        }
        \foreach \i in {1,2,3}{\foreach \j in {4,5,6}{
          \draw[dashed] (u\i) -- (v\j) (v\i) -- (u\j);
        }}
        \draw[dashed] (u2) -- (v1);
    \end{tikzpicture}
    \caption{The graph has exactly~2 non-trivial components, with at least one non-edge in one of them.}
    \label{subfig:non-complete-component}
  \end{subfigure}

  \caption{%
    Impossible situations in the proof of \cref{lem:bipartite-disco-classification},
    in which enough non-edges (drawn with dashes) are found for~$\overline{B}$ to be connected.
  }
\end{figure}
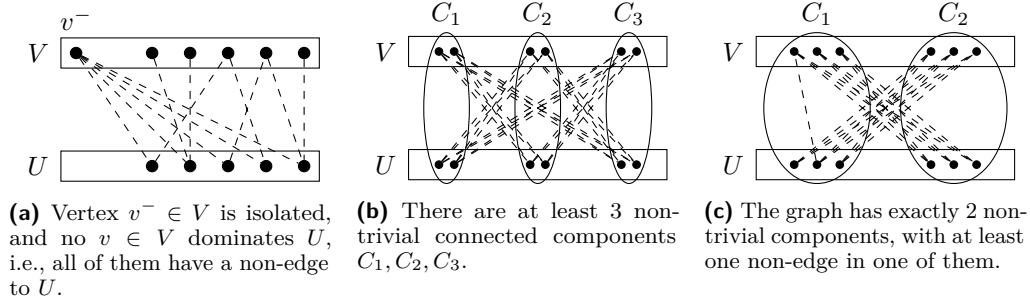

Next, we prove \cref{lem:informal} for bipartite graphs when the partition~$\PP$ is trivial.
In this context, a flip of a bipartite graph~$B = (U,V,E)$ should preserve the fixed bipartition~$(U,V)$,
i.e., we only consider flips between some subset of~$U$ and some subset of~$V$.
We call this a \emph{bipartite flip}.
\begin{lemma}
  \label{lem:informal-bipartite}
  For every bipartite graph~$B = (U,V,E)$, there are partitions~$(U_1,U_2)$ of $U$ and~$(V_1,V_2)$ of $V_2$,
  and a $\{U_1,U_2,V_1,V_2\}$-bipartite flip~$B'$ such that if~$uv \in E(B')$, then~$\dist_B(u,v) \le \diamBoundBip$ and $\dist_{\overline{B}}(u,v) \le \diamBoundBip$.
  Furthermore, we either have $U_2=\emptyset$ or $U_1 = N(v)$ for some~$v \in V$; and similarly, either $V_2=\emptyset$ or $V_1=N(u)$ for some $u\in U$.
\end{lemma}
\begin{proof}
  We apply \cref{lem:bipartite-disco-classification} to choose the appropriate partition:
  \begin{enumerate}
    \item If~$B$ or~$\overline{B}$ is connected, then the partition is trivial: $U_1 = U, U_2=\emptyset, V_1 = V, V_2=\emptyset$.
    \item If~$B$ is the disjoint union of two bicliques, then we partition accordingly so that the bicliques are~$(U_1,V_1)$ and~$(U_2,V_2)$.
    \item If~$v^-,v^+ \in V$ are isolated and fully adjacent to~$U$ respectively,
      then we pick~$u \in U$ arbitrarily and let~$V_1 = N(u)$ and $V_2 = V \setminus V_1$. The partition of~$U$ remains trivial: $U_1 = U$ and $U_2=\emptyset$.
  \end{enumerate}
  In all three cases, for each nonempty~$U_i,V_j$ the bipartite graph~$B[U_i,V_j]$ or its complement is connected.
  Furthermore, the partitions~$(U_1,U_2)$ and~$(V_1,V_2)$ are of the required form.

  For all nonempty~$U_i,V_j$, \cref{lem:diam-complement-bipartite} gives that either~$B[U_i,V_j]$ or its complement has diameter at most~\diamBoundBip.
  In the former case, we flip the adjacency relation between~$U_i$ and~$V_j$,
  and the flip~$B'$ is the one obtained by doing so for each pair~$U_i,V_j$.
  By construction, $\overline{B'[U_i,V_j]}$ has diameter at most~\diamBoundBip~for all nonempty~$U_i,V_j$.
  Thus, if~$uv \in E(B')$ with~$u \in U_i$, $v \in V_j$, then~$u,v$ are at distance at most~\diamBoundBip~in both~$B'[U_i,V_j]$ and its complement.
  Finally, $B$ contains either~$B'[U_i,V_j]$ or~$\overline{B'[U_i,V_j]}$ as an induced subgraph,
  hence~such $u,v$ are also at distance at most~\diamBoundBip~in both~$B$ and~$\overline{B}$.
\end{proof}

Having established the above, we proceed with the proof of \Cref{lem:informal}. 

\begin{proof}[Proof of \cref{lem:informal}]
  Let~$G$ be a graph and~$\PP$ a partition of~$V(G)$. We construct a flip~$G'$ of~$G$ as follows:
  \begin{itemize}
    \item For each part~$X \in \PP$, we flip~$(X,X)$ if necessary to ensure that~$\diam(\overline{G'[X]}) \le 3$ has diameter at most~\diamBound, using \cref{lem:diam-complement}.
    \item For each pair of distinct parts~$X,Y \in \PP$, we apply \cref{lem:informal-bipartite} to the bipartite graph $G[X,Y]$.
      This gives partitions~$\PP_{X,Y}$ of~$X$ and~$\PP_{Y,X}$ of~$Y$ into at most two parts each, and a $(\PP_{X,Y} \cup \PP_{Y,X})$ bipartite flip~$G'_{XY}$ of~$G[X,Y]$, ensuring the following:
      whenever~$uv$ is an edge in~$G'_{XY}$, then~$u,v$ are at distance at most~\diamBoundBip in both~$G[X,Y]$ and~$\overline{G[X,Y]}$.
      We apply the same flip in~$G$ to obtain~$G'$.
  \end{itemize}
  Define~$\PP'$ to be the partition of~$V(G)$ whose parts are of the form $\bigcap_{Y \in \PP} P_{X,Y}$
  for all possible choices of~$X \in \PP$ and~$\{P_{X,Y} \in \PP_{X,Y}\}_{Y \in \PP}$.
  Clearly, the partition~$\PP'$ refines both~$\PP$ and all~$\PP_{X,Y}$,
  i.e., any part of the latters can be obtained as union of parts of~$\PP'$.
  It follows that~$G'$ is a $\PP'$-flip of~$G$.

  Consider an edge~$uv \in E(G')$, and let~$X,Y$ be the parts of $\PP$ containing~$u$ and~$v$, respectively (possibly $X=Y$).
  Further, consider an arbitrary $\PP$-flip~$G''$ of~$G$.
  If~$X \neq Y$, notice that~$G'[X,Y]$ is equal to the flip~$G'_{XY}$ provided by \cref{lem:informal-bipartite}.
  Thus,~$uv$ is an edge in~$G'_{XY}$, which implies that~$u,v$ are at distance at most~$\diamBoundBip$ in both~$G[X,Y]$ and~$\overline{G[X,Y]}$.
  Since~$G''$ is a $\PP$-flip of~$G$ and~$X,Y \in \PP$, the restriction~$G''[X,Y]$ coincides with either~$G[X,Y]$ or~$\overline{G[X,Y]}$, hence $\dist_{G''[X,Y]}(u,v) \le \diamBoundBip$, and a fortiori $\dist_{G''}(u,v) \le \diamBoundBip$.
  And in the case~$X = Y$, we have $\dist_{G'[X]}(u,v)=1$ and $\dist_{\overline{G'[X]}}(u,v)\leq 3$ due to $\diam(\overline{G'[X]})\leq 3$, hence also $\dist_{G''}(u,v)\leq 3$ because $G''[X]$ is equal to $G'[X]$ or its complement.
  Since~$G''$ was chosen arbitrarily as a $\PP$-flip of~$G$, this proves that $\dist_\PP(u,v) \le \diamBoundBip$ holds for any edge~$uv \in E(G')$.
  So we have more generally $\dist_{\PP}(u,v) \le \diamBoundBip\cdot \dist_{G'}(u,v)$ for every pair of vertices $u,v$,
  which in terms of balls is equivalent to $\Ball^{r}_{G'}(v) \subseteq \Ball^{6r}_{\PP}(v)$ for all vertices~$v$ and radius~$r$.

  Consider now the partition~$\PP'$ used to obtain the flip~$G'$.
  Recall that its part are of the form $\bigcap_{Y \in \PP} P_{X,Y}$
  for all~$|\PP|$ choices of~$X \in \PP$, and all~$2^{|\PP|}$ choices of $\{P_{X,Y} \in \PP_{X,Y}\}_{Y \in \PP}$.
  This immediately gives the general bound $|\PP'| \le |\PP| \cdot 2^{|\PP|}$.
  Assuming now that~$G$ has VC-dimension at most~$d$,
  we will improve this bound using the additional restriction on~$\PP_{X,Y}$ ensured by \cref{lem:informal-bipartite}:
  either~$\PP_{X,Y} = \{X\}$ is trivial, or there is some vertex~$v_{X,Y} \in Y$ such that $\PP_{X,Y} = \{X \cap N(v_Y), X \setminus N(v_Y)\}$ is the partition into neighbours and non-neighbours of~$v_Y$.
  Call~$A$ the collection of these vertices~$v_Y$.
  The above implies that~$\PP'$ partitions~$X$ into neighbourhood types over~$A$,
  that is parts of the form $\{x \in X : N(x) \cap A = B\}$ for all choices of~$B \subseteq A$.
  By the Sauer--Shelah lemma (see~\cref{sauershelah}), if~$G$ has VC-dimension at most~$d$,
  then at most $\Oof(|A|^d) = \Oof(|\PP|^d)$ of these neighbourhood types are non-empty.
  Summing over all choices of~$X$, this gives $|\PP'| = \Oof(|\PP|^{d+1})$.
  
  We finally compute the running time of this procedure. For each part $X \in \PP$, the algorithm checks if $\diam(G[X])\leq 3$ and either complements or keeps $G[X]$ as is; evidently, this can be achieved in time $\cal{O}(|\PP|\cdot |V(G)|^2)$. Additionally, for each distinct pair of parts $X,Y \in \PP$, the algorithm checks if $\diam(G[X,Y])\leq 6$ or $\diam(\overline{G[X,Y]})\leq 6$; if neither holds, then $G[X,Y]$ and $\overline{G[X,Y]}$ are necessarily both disconnected by \Cref{lem:diam-complement-bipartite}, and so the algorithm checks if $G[X,Y]$ is the disjoint union of two bicliques. This may be achieved in time $\cal{O}(|\PP|^2\cdot |V(G)|^2)$. In any one of these outcomes, the flips and the refinement of the partition are directly obtainable by \Cref{lem:informal-bipartite}.
\end{proof}

We remark that an inspection of the proof of \Cref{lem:inc-lemma} yields a polynomial-time algorithm for converting arbitrary flips into definable ones. The combination of this fact together with \Cref{lem:informal} gives an effective lemma for translating the partition metric into the metric induced by a definable flip, i.e., an effective version of \Cref{thm:conversion}.

%% file: residuality.tex
\section{Flip-separability}\label{sec:flip-separability}

The starting point to our proof of separability is the recent characterization of monadic dependence via flip-breakability \cite{flip-breakability}, a combinatorial property that allows to find large sets that are pairwise far apart after a bounded number of flips. To make use of this property, we require some additional tools which we believe will be very useful in the analysis of monadically dependent classes. The first is a locality property of the partition metric.

\input{locality}

\input{flip-breakability.tex}

A simple application of flip-breakability is the following:
in a graph~$G$ from a monadically dependent class with weights~$\weight$, given a very large set of vertices~$W$,
there is a large subset~$A \subset W$ and a bounded set~$S$ satisfying~$\weight(\Ball_S^r(A)) \le \eps \cdot \weight(V(G))$.
Indeed, for~$p \coloneqq \ceil{\frac{1}{\eps}}$, by applying~$p-1$ times flip-breakability with radius~$2r$ to~$W$,
we obtain~$p$ large subsets $A_1,\dots,A_p$ of~$W$ pairwise at distance more than~$2r$ in some flip metric~$\dist_S$ with bounded~$|S|$.
Then, the $r$-balls~$\Ball_S^r(A_i)$ are all disjoint, and one of them must have only an $\eps$-proportion of the total weight.

The main technical result of this section (\cref{lem:small-balls} below) is a variant of this statement in which the sets~$W$ and~$S$ of vertices are replaced by a families~$\FF$ and~$\YY$ of $t$-tuples of vertices.
To state this result, we need to define the flip-metric associated with~$\YY$, which combines the flip metric~$\dist_S$ for all~$S \in \YY$.
For a graph $G$ and a family $\YY \subseteq 2^{V(G)}$, we define for all $u,v\in V(G)$, and $A,B \subseteq V(G)$,
\[
    \dist_\YY(u,v) := \max_{S \in \YY} \dist_S(u,v)
    \quad
    \text{and}
    \quad
    \dist_\YY(A,B) := \min_{a\in A, b\in B} \dist_\YY(a,b).
\]
Thus, we also have
\[
    \Ball_\YY^r(v) := \bigcap_{S \in \YY} \Ball_S^r(v)
    \quad
    \text{and}
    \quad
    \Ball_\YY^r(A) = \bigcup_{v\in A} \Ball_\YY^r(v).
\]

We call $\FF$ \emph{$t$-uniform} if every set in $\FF$ has size exactly~$t$. We will use the sunflower lemma:
\begin{theorem}[\cite{sunflower}]\label{lem:sunflower}
  For every $t \in \N$ there is a function $f:\N \to \N$ satisfying the following.
  For every $m \in \N$ and in any $t$-uniform family $\FF$ of size $f(m)$, there is a~subfamily $\FF' \subset \FF$ of size~$m$ and a set~$Y$ such that for all~$A \neq B$ in~$\FF'$, $A \cap B = Y$.
  The subfamily~$\FF'$ is called a sunflower, and~$Y$ its core.
\end{theorem}

% \begin{claim}
%     $v \in \Ball_\YY^{r_1}(A) \implies \Ball_\YY^{r_2}(v) \subseteq \Ball_\YY^{r_1 + r_2}(A)$.
% \end{claim}
% \begin{claimproof}
%     Assume $u \in \Ball_\YY^{r_2}(v)$.
%     Then in every $S \in \YY$ flip, there is a $uv$ path of length at most $r_2$.
%     Since $v \in \Ball_\YY^{r_1}(A)$, there must be $a \in A$ with $v \in \Ball_\YY^{r_1}(a)$.     
%     Then in every $S \in \YY$ flip, there is a $va$ path of length at most $r_1$.
%     Composing the path gives us a $va$ path of length $r_1 + r_2$ in every $S \in \YY$ flip.
%     Therefore, $v \in \Ball_\YY^{r_1 + r_2}(a)$ and also $v \in \Ball_\YY^{r_1 + r_2}(A)$.
% \end{claimproof}

% Before we show our main result, we introduce some notation.
% For a~subset $S$ of vertices of a~graph~$G$ and a~set $R \subseteq \cal P_S^2$ (see~\cref{sec:flip-def} for the definition of $\cal P_S$), we denote by $G \oplus (S, R)$ the flip of $G$ obtained after making the $(A,B)$-flip for every $(A,B) \in R$.
% We also set $G \oplus (\bar s, R) := G \oplus (S, R)$ where $S$ is the set of vertices appearing in tuple~$\bar s$.

% Let $\HH$ be a set of graphs with the same vertex set $V$.
% For all $v \in V$ and $r\in \N$, we define
%     \[
%         B_\HH^r(v) := \bigcap_{H \in \HH} B_H^r(v).
%         \]
%     We also set $\dist_\HH(u,v) := \max\limits_{H \in \HH} \dist_H(u,v)$.    
%     For a tuple $\bar s \in V^t$ we write 
%     \[
%         B_\HH^r(\bar s) := \bigcup_{i \in [t]} B_\HH^r(\bar s[i]).
%     \]    

We proceed with the main technical lemma of this section.

\begin{lemma}\label{lem:small-balls}
    For every monadically dependent graph class $\CC$,  radius $r \in \N$, and $\eps >0$, there are numbers $t := t(\CC,r)$ and $k := k(\CC, r, \eps)$ and a function $M := M(\CC,r,\eps) : \N \to \N$ satisfying the following.
    Given a graph $G\in\CC$ with weights $\weight$ and a $t$-uniform family $\FF \subseteq 2^{V(G)}$ of size $|\FF| \geq M(m)$,
    there exist two $t$-uniform families $\FF' \subseteq \FF$ and $\YY \subseteq 2^{V(G)}$ of size $|\FF'| \geq m$ and $|\YY| \leq k$,
    such that for each $S \in \FF'$, we have
    \[
        \weight(\Ball^r_\YY(S- \textstyle{\bigcup \YY})) \leq \eps \cdot \weight (V(G)).
    \]
\end{lemma}

The proof follows the previously sketched argument, replacing vertices with $t$-tuples.
We repetitively apply flip-breakability to~$\FF$, considering each coordinate of $t$-tuples one after the other,
until~$\FF$ is broken into~$\ceil{\frac{1}{\eps}}$ subsets with pairwise disjoint $r$-neighbourhoods.
One of these subsets has a $r$-neighbourhood with only an $\eps$-fraction of the total weight, and satisfies the desired condition.
The sunflower lemma is used as a preprocessing step to reduce the problem to the case where tuples in~$\FF$ are pairwise disjoint.
\begin{proof}[Proof of \cref{lem:small-balls}]
    Fix~$\CC,r,\eps$ as in the statement, choose~$t := t(\CC,r)$ as given by flip-breakability (\cref{thm:flip-breakability}),
    and fix $p \coloneqq \ceil{\frac{1}{\eps}}$, $\ell \coloneqq \binom{p}{2} t^2$, and $k \coloneqq k(\CC,r,\eps) = \ell + 1$.
    Further, call~$M_{\text{brk}} : \N \to \N$ the function given by \cref{thm:flip-breakability} (depending on~$\CC,r$),
    and define \[M(m) \coloneqq f(p \cdot M_{\text{brk}}^{(\ell)}(m)),\] where~$f$ is the bound in the sunflower lemma (\cref{lem:sunflower}), and~$M_{\text{brk}}^{(\ell)}$ is the $\ell$-fold composition of~$M_{\text{brk}}$.

    Consider now~$G \in \CC$ with weights~$\weight$ and a $t$-uniform family~$\FF$ of size~$M(m)$.
    Applying \cref{lem:sunflower} to~$\FF$, we obtain a subfamily $\SS \subset \FF$ of size~$p \cdot M_{\text{brk}}^{(\ell)}(m)$ and a \emph{core}~$Y$ such that for all~$A \neq B \in \SS$, $A \cap B = Y$.
    Define $\SS' = \{S \setminus Y~\colon~S \in \SS\}$ the collection of \emph{petals} of the sunflower~$\SS$.
    By construction, it is $t'$-uniform for $t' := t - |Y|$, and no two sets in~$\SS'$ share a vertex.
    We arbitrarily partition~$\SS'$ into~$p$ subfamilies $\FF_1,\dots,\FF_p$, each of equal size~$M_{\text{brk}}^{(\ell)}(m)$.

    We now write each set in each~$\FF_q$ as a $t'$-tuple $\bar{s} = (s_1,\dots,s_{t'})$.
    Our goal, using \cref{thm:flip-breakability}, is to restrict each~$\FF_q$ to some subfamily of size~$m$,
    so that~$\bigcup \FF_q$ and~$\bigcup \FF_{q'}$ are far apart in some appropriate flip metric for~$q \neq q'$.
    Note here that we allow ourselves to discard entire tuples from the family~$\FF_q$, but we cannot discard individual vertices from a~kept tuple $\bar{s} \in \FF_q$.

    We proceed as follows for every~$q < q'$ and every $i,j \in [t']$:
    define $W_1 = \{s_i \colon \bar{s} \in \FF_q\}$ the set of $i$th vertices in tuples of~$\FF_q$, and similarly $W_2 = \{s_j \colon \bar{s} \in \FF_{q'}\}$.
    Since the sunflower lemma ensures that no two tuples of~$\FF_q$ share a vertex,
    we have $|W_1| = |\FF_q|$ and $|W_2| = |\FF_{q'}|$, which at the start of the procedure are equal to $M_{\text{brk}}^{(\ell)}(m)$.
    Flip-breakability (\cref{thm:flip-breakability}) applied to $W_1, W_2$ then gives subsets $A_1 \subset W_1$, $A_2 \subset W_2$ of size~$M_{\text{brk}}^{(\ell-1)}(m)$ each,
    and a set~$Y_1$ of size~$t$ such that $\Ball_{Y_1}^r(A_1) \cap \Ball_{Y_1}^r(A_2) = \emptyset$, or equivalently $\dist_{Y_1}(A_1,A_2) > 2r$.
    We then restrict~$\FF_q$ to the~$M_{\text{brk}}^{(\ell-1)}(m)$ tuples whose~$i$th element belongs to~$A_1$, and $\FF_{q'}$ to the~$M_{\text{brk}}^{(\ell-1)}(m)$ tuples whose~$j$th element belongs to~$A_2$.
    For simplicity, we also restrict every remaining~$\FF_{q''}$, $q'' \not\in \{q,q'\}$ to an~arbitrary subfamily of size~$M_{\text{brk}}^{(\ell-1)}(m)$.
    We then continue with the next choice of~$q,q',i,j$, yielding the next set $Y_2$ of size $t$, and further restricting the families~$\FF_s$ to size~$M_{\text{brk}}^{(\ell-2)}(m)$.

    After $\ell = \binom{p}{2} t^2$ steps, $m$ tuples remain in each~$\FF_q$, and we have accumulated flip-defining sets $Y_1,\dots,Y_\ell$, each of size~$t$.
    Adding the sunflower core, define $\YY = \{Y,Y_1,\dots,Y_\ell\}$
    (here~$|Y| < t$, but we can add arbitrary elements to~$Y$ to make~$\YY$ $t$-uniform: this only helps to reach the conclusion of the lemma).
    By construction, for all $q \neq q'$ and all~$i,j \in [t']$, the~$i$th elements of tuples of~$\FF_q$ are at distance more than~$2r$ from the~$j$th elements of tuples of~$\FF_{q'}$ in the flip metric~$\dist_{Y_{\ell'}}$ for some~$\ell' \in [\ell]$.
    A fortiori, the same holds in the metric~$\dist_\YY$.
    Therefore, $\dist_\YY\left(\bigcup \FF_q, \bigcup \FF_{q'}\right) > 2r$ for all~$q \neq q'$,
    or equivalently $\Ball_\YY^r(\bigcup \FF_q)$ and $\Ball_\YY^r(\bigcup \FF_{q'})$ are disjoint.
    Since these $p = \ceil{\frac{1}{\eps}}$ balls are disjoint, the one with minimum weight, say around~$\FF_q$, will satisfy
    \begin{equation}
        \weight(\Ball_\YY^r(\textstyle{\bigcup \FF_q})) \leq \eps \cdot \weight(V(G)).
        \label{eq:tuples-weight-bound}
    \end{equation}
    Finally, we construct~$\FF'$ by adding back the core~$Y$ to each tuple of~$\FF_q$, so that they are once again $t$-tuples.
    Since~$Y \in \YY$, we clearly have for any~$S \in \FF'$ that $S - \bigcup \YY \subseteq \bigcup \FF_q$,
    and thus the conclusion follows directly from \eqref{eq:tuples-weight-bound}.
\end{proof}

\subsection{Monadic dependence implies flip-separability}
We are now ready to prove our main result:
any monadically dependent (or equivalently flip-breakable) class~$\CC$ is flip-separable.
The core of the proof is the following lemma, in which flip-separability is modified to use the metric~$\dist_\FF$ for a $t$-uniform family~$\FF$, as defined in the previous section.
For simplicity, we assume that no vertex has large weight:
a~weight function $\weight \colon V(G) \to \R_{\geq 0}$ of a~graph $G$ is called \emph{$\eps$-balanced} if no vertex $v \in V(G)$ has weight larger than $\eps \cdot \weight(V(G))$.
Vertices with large weight will be added back afterwards.

\begin{lemma}\label{lem:seplargeweight}
    For every monadically dependent graph class $\CC$, $r \in \N$, and $\eps > 0$, there are $t,k \in \N$ with the following property.
    For every $G \in \CC$ with an $\eps$-balanced weight function $\weight$, there exists a
    family $\FF \subseteq 2^{V(G)}$ of  at most $k$ many sets each of size at most~$t$, such that for every vertex $v \in V(G)$, we have
    \[\weight(\Ball_\FF^r(v)) \leq \eps \cdot \weight(V(G)).\]
\end{lemma}

% \begin{theorem}
%     Let $\CC$ be a monadically dependent graph class, $r \in \N$, and $\eps > 0$, there are $t,k \in \N$ with the following property.
%     For every $G \in \CC$, there exists a set $\HH$ of $k$ many $t$-definable flips of $G$ and a~set $Z \subseteq V(G)$ of size at~most~$\lceil \frac{1}{\varepsilon} \rceil$ such that \[\weight(B_\HH^r(v)) \leq \eps \cdot \weight(V(G))~~\text{for every~}v \in V(G) \setminus Z.\]
% \end{theorem}

We will call a $t$-uniform family $\FF \subseteq 2^{V(G)}$ a \emph{sparsifying family} if it satisfies the conclusion of \cref{lem:seplargeweight}, i.e.\ for every vertex $v \in V(G)$,
\[\weight(\Ball_\FF^r(v)) \leq \eps \cdot \weight(V(G)).\]

The proof of \cref{lem:seplargeweight} follows the sketch given at the end of \cref{sec:intro}.
The goal is to find a small sparsifying family~$\FF$.
We start with a very large~$\FF$: for instance any~$\FF$ satisfying $\bigcup \FF = V(G)$ is trivially sparsifying.
As long as~$\FF$ is very large, we shrink it as follows.
Applying \cref{lem:small-balls} gives a large subfamily~$\FF' \subset \FF$ and a $t$-uniform family~$\YY$ of bounded size which sparsifies~$\FF'$.
This ensures that any vertex close to~$\FF'$ is sparsified by~$\YY$ itself.
For vertices which are far from~$\FF'$, we use locality (\cref{lem:partition-gaifman}) to show that a bounded subfamily $\overline \XX \subset \FF'$ is enough to replicate the sparsifying power of~$\FF'$.
Then, tuples in $\XX \coloneqq \FF' \setminus \overline \XX$ are redundant, and $(\FF - \XX) \cup \YY$ is a sparsifying family.
Since~$|\FF'|$ is unbounded (function of~$|\FF|$) and~$|\overline \XX|,|\YY|$ are bounded, this new sparsifying family is smaller than~$\FF$.

\begin{proof}[Proof of \cref{lem:seplargeweight}]
Fix $\CC$, $r$, $\eps$ as in the statement.
Let $r' > r$  be the distance given by \cref{lem:partition-gaifman} for formulas with quantifier rank at most $r$.
Let $t \coloneqq t(\CC,3r',\eps)$, $\lambda \coloneqq k(\CC,3r',\eps)$, and $M \coloneqq M(\CC,r,\eps)$ be defined as in~\cref{lem:small-balls}.
% Let $d$ be the VC-dimension bound for $\CC$ and
% $\ell := t(r,2,t,t^\lambda,d)$ be the number of colors given by \cref{lem:partition-gaifman}.
All defined quantities so far depend only on $\CC$, $r$, $\eps$.

%\mipiin{The case $|V(G)|<t$ does not work, balls may have weight $\eps t$, not $\eps$.
%Colin: I don't get the issue. $\eps$-balancedness says vertices have weight at most~$\eps \cdot \weight(V(G))$, i.e.\ exactly what we want, and the flip just deletes all edges.}
Fix a graph $G \in \CC$ with $\eps$-balanced weights $\weight$.
We may assume that~$G$ has at least~$t$ vertices, as otherwise $\eps$-balancedness lets us conclude with $\FF \coloneqq \{V(G)\}$.
Since the weights~$\weight$ are $\eps$-balanced, there exists a (large) sparsifying family $\FF \coloneqq \{ S_v~\colon~v \in V(G) \}$ where we can choose $S_v$ to be any set of size $t$ that contains~$v$.
We will show that there is a $k\in\N$, which only depends on $\CC,r,\eps$, such that whenever we have a sparsifying family $\FF$ of size at least $k$, we can also find a sparsifying family $\FF^*$ of strictly smaller size $|\FF^*| < |\FF|$.
This will prove the lemma by induction.

We will specify the value of $k$ later, and now assume that we are given a sparsifying family~$\FF$ of size at least $k$ that we want to compress.
We first apply \cref{lem:small-balls} with radius~$3r'$ and~$\eps$ to $\FF$ which yields $t$-uniform families $\FF' \subseteq \FF$ and $\YY$ with $|\FF'| \geq M^{-1}(k)$ and $|\YY| \leq \lambda$ such that for each $S \in \FF'$,
\begin{equation*}\label{eq:small-balls}
    \weight(\Ball^{3r'}_\YY(S- \textstyle{\bigcup \YY} )) \leq \eps \cdot \weight (V(G)).\tag{$*$}
\end{equation*}
We will next show the existence of a family $\XX \subseteq \FF'$ such that
\[
    \FF^* := (\FF - \XX) \cup \YY
\]
is a sparsifying family that is smaller than $\FF$.
We first show that vertices $v$ \emph{close to $\FF'$} already define balls of small-enough weight, regardless of~$\XX$.
This is made formal with the following claim.
\begin{claim}\label{clm:balls-close-to-S}
    For every $v \in \Ball_\YY^{2r'}(\bigcup \FF')$ we have $\weight(\Ball_\YY^{r}(v)) \leq \eps \cdot \weight (V(G))$.
\end{claim}
\begin{claimproof}
    Consider a vertex~$v$ as in the claim, i.e.\ $v \in \Ball_\YY^{2r'}(S)$ for some $S \in \FF'$.
    By definition, every $y \in \bigcup\YY$ satisfies $\Ball^{2r'}_\YY(y) = \Ball^{r}_\YY(y) = \{y\}$.
    This means if $v \in \Ball^{2r'}_\YY(\bigcup \YY)$ then actually $v \in \bigcup\YY$ and $\Ball^r_\YY (v) = \{ v \}$ has small weight by $\eps$-balancedness. 
    Otherwise, we must have $v \in \Ball^{2r'}_\YY(S - \bigcup \YY)$.
    Then by the triangle inequality
    \[
        \Ball^{r}_\YY(v) \subseteq \Ball^{r'}_\YY(v) \subseteq \Ball^{3r'}_\YY(S- \textstyle{\bigcup \YY})
    \]
    and $\Ball_\YY^{r}(v)$ has small weight by \eqref{eq:small-balls}.
\end{claimproof}
We will next use the locality of first order logic, to show that we can handle the vertices \emph{far from $\FF'$}, by only keeping few representatives from $\FF'$.
We define the family $\XX$ containing the ``redundant'' sets from $\FF'$.
In order to build $\XX$,
consider the first-order formula  $\phi(vw, \bar s)$ with quantifier rank at most $r$ that asserts $\dist_S(v,w) > r$ for every two vertices $v,w \in V(G)$, set~$S$ of size $t$, and $t$-tuple $\bar s$ that enumerates $S$.
Moreover, let $\PP$ be the coarsest simultaneous refinement of every partition $\PP_S$ with $S \in \YY$.
Each $\PP_S$ has size at most~$(t+2^t)$, hence the size of $\PP$ is bounded by $p \coloneqq (t+2^t)^\lambda$,
and $\dist_\PP(u,v) \geq \dist_\YY(u,v)$ holds for all vertices $u,v \in V(G)$.

Applying \cref{lem:partition-gaifman} to $\phi$ and $p$ yields a number $\ell \coloneqq \ell(r,2,t,p)$ that only depends on $\CC,r,\eps$, and two $\ell$-colorings $\col_1 : V(G)^2 \rightarrow [\ell]$ and $\col_2 : V(G)^t \rightarrow [\ell]$ of the $2$-tuples and $t$-tuples of $V(G)$ such that for all $uv \in V(G)^2$ and $\bar s \in V(G)^t$ with $\dist_\PP(uv, \bar s)  > r'$, whether or not $G\models \phi(uv,\bar s)$ holds only depends on $\col_1(uv)$ and $\col_2(\bar s)$.
By our choice of $\phi$, whether or not $\dist_S(u,v) > r$ holds for $S := \{ s : s\in \bar s \}$ depends only on these colors, too.
Since $\dist_\PP(uv, \bar s) \geq \dist_\YY(uv, \bar s)$, the above in particular holds when $\dist_\YY(uv, \bar s) > r'$.

Let $\overline \XX \subseteq \FF'$ be a subfamily constructed by picking for every color $K \in [\ell]$ a set $S \in \FF'$ such that $\col_2(\bar s) = K$ holds for some enumeration $\bar s$ of $S$ (when such a set~$S$ exists for~$K$).
We set $\XX := \FF' - \overline{\XX}$, which completes the definition of $\FF^*$.
Let us now show that $\FF^*$ is a sparsifying family.
We have already argued in \cref{clm:balls-close-to-S} that vertices close to some set $S \subseteq \FF'$ have balls of sufficiently small weight.
It remains to argue that the same holds for vertices far from~$\FF'$.

\begin{claim}\label{clm:balls-far-from-S}
    For every $v \in V(G)$ with $v \notin \Ball^{2r'}_\YY(\bigcup \FF')$, we have $\Ball^r_{\FF^*}(v) \subseteq \Ball^r_{\FF}(v)$.
\end{claim}

\begin{claimproof}
    Let $v$ be as in the statement of the claim.
    % Rephrasing the condition we have $v \notin \Ball^{2r'}_\YY(\bigcup \FF')$.
    Let $w \notin \Ball^r_{\FF}(v)$. 
    We want to show that $w \notin \Ball^r_{\FF^*}(v)$, too.
    By assumption, there exists a set $S \in \FF$ such that $\dist_S(v,w) > r$.
    If $S$ is also contained in $\FF^*$, then we are done.
    Otherwise, $S \in \XX \subseteq \FF'$.
    Let $\bar s$ be a $t$-tuple that enumerates $S$.
    We know that $G \models \phi(vw,\bar s)$.
    Since $S \in \XX \subseteq \FF'$, we know that $\dist_\YY(v, \bar s) > 2r'$ by assumption.
    Now if $\dist_\YY(v,w) > r$, then we are also done.
    Otherwise, we have $\dist_\YY(v,w) \leq r$ and by the triangle inequality $\dist_\YY(vw, \bar s) > r'$.
    By construction of~$\XX$ there is another $S' \in \FF^*$ that has an enumeration $\bar s'$ such that $\col_2(\bar s) = \col_2(\bar s') = K$.
    By the same argument as before, also $\dist_\YY(vw, \bar s') > r'$.
    As $\bar s$ and $\bar s'$ have the same color and are both far from $vw$, \cref{lem:partition-gaifman} gives us that also $G \models \phi(vw,\bar s')$ and $\dist_{\FF^*}(v,w) > r$, as desired. 
\end{claimproof}
It follows by \cref{clm:balls-close-to-S}, \cref{clm:balls-far-from-S}, and $\FF$ being a sparsifying family, that also $\FF^*$ is a sparsifying family.
It finally remains to analyze the size of $\FF^*$.
We have 
\[
    |\FF^*| \leq |\FF| - |\XX| + |\YY|,
\]
where $\FF$ has size at least $k$, $\YY$ has size at most $\lambda$, and $\XX$ has size at least $M^{-1}(k) - \ell$.
This means 
\[
    |\FF^*| \leq |\FF| - M^{-1}(k) + \ell + \lambda.
\]
Setting $k \coloneqq M(\ell + \lambda + 1)$ yields $|\FF^*| < |\FF|$ as desired.
\end{proof}

We can now wrap things up in the main result of this section.

\begin{lemma}\label{lem:nip-implies-sep}
   Every monadically dependent graph class is flip-separable.
\end{lemma}

\begin{proof}
    Let $\C$ be monadically dependent and fix $r \in \N$ and $\varepsilon > 0$. Consider the graph class 
    \[ \CC^+:=\{G + I_n : G \in \CC, n \in \N\} \] 
    where $I_n$ denotes an independent set of size $n$, and~$+$ denotes disjoint union of graphs. Evidently, $\CC^+$ is still monadically dependent. 
    Let $t,k_0 \in \N$ be the constants obtained by applying \Cref{lem:seplargeweight} on $\CC^+$ with $r':=6r$ and $\eps':=\eps$. We fix $k_1:= (k_0t + 2^{k_0t})$, $k_2:=k_1\cdot 2^{k_1}$, and $k:=k_2\cdot(\floor{\frac{1}{\varepsilon}}+2^{\floor{\frac{1}{\varepsilon}}})$. Take $G \in \C$ and weights $\weight: V(G) \to \R_{\geq 0}$, and let $W:= \{v \in V(G): \weight(v)>\eps\cdot \weight(V(G))\}$ be the set of vertices of $\eps$-large weight. Evidently, $|W|\leq \floor{\frac{1}{\eps}}$. By assigning $\eps$-small weights to vertices in the independent set we obtain some $c \in \N$, a graph $G^+:=G + I_c \in \CC^+$, and a weight function $\weight^+:V(G^+) \to \R_{\geq 0}$ with the following properties:
    \begin{enumerate}
        \item $\weight^+(v)=\weight(v)$ for all $v \in V(G)\setminus W$;
        \item $\weight^+(v)=0$ for all $v \in W$;
        \item $\weight^+(V(G^+))=\weight(V(G))$;
        \item $\weight^+$ is $\eps$-balanced.
    \end{enumerate}
    It follows by \Cref{lem:seplargeweight} that there is a collection $\FF\subseteq 2^{V(G)}$ of at most $k_0$ many sets of size at most $t$ such that for every vertex $v \in V(G^+)$
    \[ \weight^+(\Ball_\FF^{6r}(v)) \leq \eps \cdot \weight^+(V(G^+)).\]
    Let $\SS = \bigcup \FF \subseteq V(G)$; clearly $|\SS|\leq k_0t$. Consider the partition $\PP$ of size $k_1$ into $\SS$-classes. By \Cref{lem:informal} we obtain a $k_2$-flip $G^+_0$ of $G^+$ such that for all $v\in V(G)$
    \[ \Ball^r_{G^+_0}(v) \subseteq \Ball^{6r}_\PP(v)=\Ball_\SS^{6r}(v)\subseteq \Ball_\FF^{6r}(v).\]
    Finally, consider the $k$-flip $G_1^+$ of $G^+$ obtained from $G^+_0$ by isolating the vertices in $V(G)\setminus W$, and let $G':=G^+_1[V(G)]$; this is a $k$-flip of $G$. It follows that for every $v \in V(G)$ of $\eps$-small weight,
    \begin{gather*}
    \weight(\Ball^r_{G'}(v))= \weight(\Ball^r_{G^+_1}(v)) = \weight^+(\Ball^r_{G^+_0}(v)) \leq \weight^+(\Ball_\FF^{6r}(v))\\ \leq \eps \cdot \weight^+(V(G^+)) = \eps \cdot \weight(V(G)).
    \end{gather*}
    Hence, $\C$ is flip-separable as claimed. 
\end{proof}

\subsection{Flip-separability implies monadic dependence}

We finally show that flip-separability implies flip-breakability, which is known to be equivalent to monadic dependence~\cite{flip-breakability}.

\begin{lemma}\label{lem:sep-implies-break}
   Every flip-separable graph class is flip-breakable, and hence monadically dependent. 
\end{lemma}
\begin{proof}
  For every integer $r \geq 0$ and for $\eps = 1/2$, let $k$ satisfy the definition of flip-separability of $\mathcal C$ for the parameters \emph{$4r$} and $\eps$.
  We show that we can take $t = t(r,\mathcal C) := k$ and $M_r(m) = 4m^2$, in~\cref{def:flip-breakable}, to witness that $\mathcal C$ is flip-breakable. 

  Fix $G \in \mathcal C$, $m \in \NN$, and $W \subseteq V(G)$ of size at~least~$M_r(m) = 4m^2$.
  We can assume without loss of generality that $W$ has size exactly $4m^2$.
  Define the weights $\weight \colon V(G) \to \R_{\geq 0}$ by $\weight(v)=1$ for every $v \in W$, and $\weight(v)=0$ otherwise.
  The flip-separability of $\mathcal C$ with radius $4r$, $\eps=1/2$, and this weight function yields a~$k$-flip $H$ of $G$ such that
  \[\weight(\Ball_H^{4r}(v)) = \big|\Ball_H^{4r}(v) \cap W\big| \leq \eps \cdot \weight(V(G))=|W|/2\]
  for every $v \in V(G)$ of weight at~most~$|W|/2$.
  We can of course assume $m > 0$ (there is nothing to prove otherwise), thus $|W| \geqslant 4$, and \emph{every} vertex of~$G$ has weight at~most~$|W|/2$.
  We distinguish two covering cases.
  
  \subparagraph*{Case 1:} There is a~vertex $v \in V(G)$ such that $|\Ball_H^{2r}(v) \cap W| \geq \sqrt{|W|}$.\\
  On the one hand, $A_1 \coloneqq \Ball^{2r}_H(v) \cap W$ has size at least $\sqrt{4m^2} = 2m > m$.
  From the flip-separability, we also know that $A' \coloneqq \Ball^{4r}_H(v) \cap W$ has size at most $|W|/2$.
  Therefore, on the other hand, $A_2 \coloneqq W \setminus A'$ has size at least $|W|/2 = 2m^2 \geq m$.
  In $H$, no vertex of $A_1$ is at distance at~most~$2r$ of a~vertex $w$ of~$A_2$, as that would put $w$ in $\Ball^{4r}_H(v)$. 
  Thus, $A_1, A_2$ satisfy the conclusion of~\cref{def:flip-breakable}.
  
  \subparagraph*{Case 2:} For every $v \in V(G)$, it holds that $|\Ball_H^{2r}(v) \cap W| < \sqrt{|W|}$.\\
  Let $\{v_1, v_2, \ldots, v_h\}$ be a~maximal (greedily-constructed) subset of $W$ such that for every $i \neq j \in [h]$, $v_i$ and $v_j$ are at distance larger than~$2r$ from each other in~$H$.
  By assumption, we have that $h \geq |W|/\sqrt{|W|} = \sqrt{|W|} = 2m$.
  Therefore, $A_1 \coloneqq \{v_1, v_2, \ldots, v_m\}$ and $A_2 \coloneqq \{v_{m+1}, v_{m+2}, \ldots, v_{2m}\}$ satisfy the conclusion of~\cref{def:flip-breakable}.
\end{proof}

\Cref{lem:nip-implies-sep,lem:sep-implies-break} jointly establish~\cref{thm:equivalence}. We remark that the flip-breakability margins $M_r(m) = 4m^2$ obtained in \cref{lem:sep-implies-break} are polynomial in $m$ and depend on neither $\CC$ nor $r$. This is a strengthening of the original bounds obtained in \cite{flip-breakability}, where $M_r$ was established as a fast-growing function with dependence on both $\CC$ and $r$.
This means our proof of \cref{thm:equivalence} can also be seen as a boosting argument for flip-breakability.

%% file: locality.tex
\subsection{Locality for various metrics}
\label{sec:locality}

%One of the main advantages of the flip-metric is that distances are expressible by a first-order formula with parameters. 
%Formally, for every $k,r \in \N$ there exists a first-order formula in the language of graphs $\delta_r(x,y,\bar z)$ with $|\bar z|=k$, such that for every graph $G$ and $S \subseteq V(G)$ of $k$ vertices enumerated as a tuple $\bar s \in V(G)^k$, we have:
%\[ G \models \delta_r(u,v,\bar s) \iff \dist_S(u,v)>r \qquad \text{for every }u,v \in V(G).\]
%Evidently, the same is true for the metric induced by a concrete definable flip. As such, these metrics enjoy numerous features, such satisfying analogues of Gaifman's locality theorem \cite{gaifman}. Due to \Cref{thm:conversion}, the same features are also enjoyed by all the different flip-metrics introduced in \Cref{sec:prelims}. This observation is crucially used in our proof of flip-separability in \Cref{sec:flip-separability} below. Towards this, we first 

Evidently, for every graph $G$ and its $k$-flip $G'$, we may definably recover the structure of~$G$ in~$G'$ by expanding the language with $k$ unary predicates that account for the flip. As such, the metric induced by a concrete flip may be used in a variant of Gaifman's locality theorem~\cite{gaifman}. Interestingly, this idea also applies to any partition metric as these can be approximated by concrete flips due to \Cref{lem:informal}. This idea is crucially used in our proof of separability; we make this precise below.
Recall the following standard corollary of Gaifman's locality theorem \cite{gaifman}, see e.g. \cite{boundedLocalCliquewidth}. Here, by a {\em{$k$-colored graph}} we mean a graph with $k$ unary predicates (colors) on it, and the distance between tuples $\bar u$ and $\bar v$ is defined as the minimum distance between any vertex present in $\bar u$ and any vertex present in $\bar v$.

\begin{theorem}
  \label{thm:gaifman-std}
  For every $k \in \N$ and every first-order formula $\phi(\bar x, \bar y)$ in the language of $k$-colored graphs,
  there exist numbers $r = r(\qr(\phi))$ and $t = t(\qr(\phi),|\bar x|,|\bar y|, k)$ such that for every $k$-colored graph $G$,
  there are colorings $\col_1:V(G)^{\bar x}\to [t]$ and $\col_2:V(G)^{\bar y}\to [t]$ satisfying the property that for any two tuples $\bar u \in V(G)^{\bar x}, \bar v \in V(G)^{\bar y}$ with $\dist_G(\bar u, \bar v) > r$, whether $\phi(\bar u, \bar v)$ holds in $G$ depends only on $\col_1(\bar u)$ and~$\col_2(\bar v)$.
\end{theorem}

We argue that a variant of the above is true if we replace $\dist_G(\bar u, \bar v)$ by $\dist_{\PP}(\bar u, \bar v)$.

\begin{lemma}\label{lem:partition-gaifman}
    For every $p \in\N$ and every first-order formula $\phi(\bar x, \bar y)$ in the language of graphs there exist numbers $\rho = \rho(\qr(\phi))$ and $\ell = \ell(\qr(\phi), |\bar x|, |\bar y|, p)$ such that for every graph~$G$ and every partition $\PP$ of $V(G)$ with at most $p$ parts, there are colorings $\col_1:V(G)^{\bar x} \to [\ell]$ and $\col_2: V(G)^{\bar y} \to [\ell]$ satisfying the property that for any  $\bar u \in V(G)^{\bar x}, \bar v \in V(G)^{\bar y}$ with $\dist_\PP(\bar u, \bar v) > \rho$, whether $\phi(\bar u, \bar v)$ holds in $G$ depends only on $\col_1(\bar u)$ and $\col_2(\bar v)$.
  \end{lemma}

\begin{proof}
    Let $k:=p\cdot 2^p$, $r := r(\qr(\phi))$ and $t := t(\qr(\phi),|\bar x|,|\bar y|,k)$ from \Cref{thm:gaifman-std}, and set $\rho:=6r$, $\ell:=t$. Consider a graph $G$ and a partition $\PP$ of $G$ into $p$ parts. By \Cref{lem:informal} it follows that there is a $k$-flip $G'$ of $G$ such that 
      \[ \dist_{\PP}(u,v)\leq 6\cdot \dist_{G'}(u,v), \qquad \text{for all } u,v \in V(G).\tag{$\star$}\label{eq:star1}\]
    Consider an expansion $\hat{G'}$ of $G'$ with $k$ unary predicates, interpreted as the parts defining this flip. It follows that the edge relation of $G$ can be defined in $\hat{G'}$ by a quantifier-free formula, and consequently, we may rewrite $\phi(\bar x, \bar y)$ into a formula $\phi'(\bar x, \bar y)$ in the language of $k$-colored graphs and with the same quantifier rank as $\phi$, that satisfies 
    \[ G \models \phi(\bar u, \bar v) \iff \hat{G'}\models \phi'(\bar u,\bar v), \qquad \text{for all }\bar u \in V(G)^{\bar x},\bar v \in V(G)^{\bar y}.\tag{$\star\star$}\label{eq:star2}\]
    Applying \Cref{thm:gaifman-std}, we obtain colorings $\col_1:V(G)^{\bar x} \to [\ell]$ and in $\col_2:V(G)^{\bar y}\to [\ell]$ satisfying the property that for any $\bar u \in V(G)^{\bar x},\bar v \in V(G)^{\bar y}$ with $\dist_{G'}(\bar u,\bar v)>r$, whether $\phi'(\bar u,\bar v)$ holds in $\hat{G'}$ only depends on $\col_1(\bar u)$ and $\col_2(\bar v)$. It thus follows by (\ref{eq:star1}) and (\ref{eq:star2}) that for any tuples $\bar u \in V(G)^{\bar x},\bar v \in V(G)^{\bar y}$ with $\dist_{\PP}(\bar u,\bar v)>\rho$, whether $\phi(\bar u,\bar v)$ holds in $G$ only depends on $\col_1(\bar u)$ and $\col_2(\bar v)$, as required.
\end{proof}

%% file: flip-breakability.tex
\subsection{Flip-breakability}

We will rely on the notion of flip-breakability, introduced in~\cite{flip-breakability}.

\begin{definition}\label{def:flip-breakable}
    A graph class $\CC$ is \emph{flip-breakable} if for every radius $r\in \N$, there is a~$t \coloneqq t(\CC,r)$ and a function $M_r : \N \to \N$, such that the following holds.
    For every graph $G\in\CC$, $m\in\N$, and set $W \subseteq V(G)$ of size at~least~$M_r(m)$, there are disjoint subsets $A_1, A_2 \subseteq W$ each of size at~least~$m$ and a $t$-flip $H$ of $G$ such that $\Ball_H^r(A_1) \cap \Ball_H^r(A_2) = \emptyset$.
\end{definition}

\begin{theorem}[\cite{flip-breakability}]
    A graph class is monadically dependent if and only if it is flip-breakable.
\end{theorem}

A natural two-set variant of this property also holds,
in which two large sets~$W_1,W_2$ both of size~$M_r(m)$ are given, and one finds the distant subsets~$A_1,A_2$ of size~$m$ with~$A_i \subseteq W_i$, see \cite[full version, Theorem~19.14]{flip-breakability}.
Combined with \cref{lem:inc-lemma} which allows to replace $t$-flips with $t$-definable flips,
we obtain the following statement which we use in the proof.
\begin{corollary}\label{thm:flip-breakability}
    For every monadically dependent graph class $\CC$ and radius $r\in \N$, there is a~$t := t(\CC,r)$ and a function $M_r : \N \to \N$, such that the following holds.
    For every graph $G\in\CC$, $m\in\N$, and sets $W_1,W_2 \subseteq V(G)$ each of size at~least~$M_r(m)$, there are disjoint subsets $A_1 \subseteq W_1, A_2 \subseteq W_2$ each of size at~least~$m$ and a~\mbox{$t$-definable} flip $H$ of $G$ such that $\Ball_H^r(A_1) \cap \Ball_H^r(A_2) = \emptyset$.
\end{corollary}